\documentclass{amsart}
\usepackage{amsxtra}
\usepackage{mathtools}
\usepackage{mathdots}
\usepackage{tensor}
\usepackage{mathrsfs}
\usepackage{verbatim}
\usepackage{graphicx}
\usepackage{calc}
\usepackage{varwidth}
\usepackage{hyperref}
\usepackage[shortalphabetic,initials]{amsrefs}
\usepackage{bdsstandard}
\usepackage{bdsthmart}

\newcounter{filler}

\numberwithin{equation}{section}

\newcommand{\aform}{\ensuremath{\langle\text{~,~}\rangle}\xspace}
\newcommand{\Fil}{\ensuremath{\mathrm{Fil}}\xspace}

\newcommand{\loc}{\ensuremath{\mathrm{loc}}\xspace}
\newcommand{\naive}{\ensuremath{\mathrm{naive}}\xspace}
\newcommand{\Rt}{\ensuremath{\mathrm{t}}\xspace}
\newcommand{\sform}{\ensuremath{(\text{~,~})}\xspace}
\newcommand{\spin}{\ensuremath{\mathrm{spin}}\xspace}

\newcommand{\undertilde}{\raisebox{0.4ex}{\smash[t]{$\scriptstyle\sim$}}}

\DeclareMathOperator{\spann}{span}
\DeclareMathOperator{\WT}{WT}

\renewcommand{\L}{\ensuremath{\mathscr{L}}\xspace}

\begin{document}
   
\renewcommand{\O}{\ensuremath{\mathscr{O}}\xspace}

\title[On the moduli description of ramified unitary local models]{On the moduli description of local models for ramified unitary groups}
\author{Brian Smithling}
\address{Johns Hopkins University, Department of Mathematics, 3400 N.\ Charles St.,\ Baltimore, MD  21218, USA}
\email{bds@math.jhu.edu}

\begin{abstract}
Local models are schemes which are intended to model the \'etale-local structure of $p$-adic integral models of Shimura varieties.  Pappas and Zhu have recently given a general group-theoretic construction of flat local models with parahoric level structure for any tamely ramified group, but it remains an interesting problem to characterize the local models, when possible, in terms of an explicit moduli problem. In the setting of local models 
for ramified, quasi-split $GU_n$, work towards an explicit moduli description was initiated in the general framework of Rapoport and Zink's book and was subsequently advanced by Pappas and Pappas--Rapoport.  
In this paper we propose a further refinement to their moduli problem, which we show is both necessary and sufficient to characterize the (flat) local model in a certain special maximal parahoric case with signature $(n-1,1)$.
\end{abstract}

\maketitle

\section{Introduction}\label{s:intro}

Local models are certain projective schemes defined over a discrete valuation ring \O.  When \O is the completion of the ring of integers of the reflex field of a Shimura variety at a prime ideal, and one has a model of the Shimura variety over \O, the local model is supposed to govern the \'etale-local structure of the Shimura model.  This allows one to reduce questions of a local nature, such as flatness or Cohen--Macaulayness, to the local model, which in practice should be easier to study than the Shimura model itself.  See \cite{prs13} for an overview of many aspects of the subject.


In \cite{pappaszhu13}, Pappas and Zhu recently gave a uniform group-theoretic construction of ``local models'' for tamely ramified groups and showed that these schemes satisfy many good properties.  They showed that their construction gives \'etale-local models of integral models of Shimura varieties in most (tame) PEL cases where the level subgroup at the residual characteristic $p$ of \O is a parahoric subgroup which can be described as the stabilizer of a lattice chain.%
\footnote{We also mention forthcoming work of Kisin and Pappas \cite{kisinpappas?}
extending this result from PEL cases to cases of abelian type.}  
In this setting, Rapoport and Zink \cite{rapzink96} had previously defined natural integral Shimura models and local models in terms of explicit moduli problems based on the moduli problem of abelian varieties describing the Shimura variety, but the resulting schemes are not always flat, as was first observed by Pappas \cite{pappas00}.  In the cases where Pappas and Zhu showed that their construction gives local models of Shimura varieties, they did so by showing that it coincides with the flat closure of the generic fiber in the Rapoport--Zink local model.  When the Rapoport--Zink local model is not already flat, it remains an interesting problem to obtain a moduli description of this flat closure: the Pappas--Zhu schemes are themselves defined via flat closure of the generic fiber, which does not impart a ready moduli interpretation.


When the group defining the Shimura variety splits over an unramified extension of $\QQ_p$ and only involves types $A$ and $C$, G\"ortz showed that the Rapoport--Zink local model is flat in \citelist{\cite{goertz01}\cite{goertz03}}.  
By contrast, the objects of study in the present paper are
local models attached to a \emph{ramified}, quasi-split unitary group, which were Pappas's original examples in \cite{pappas00} 
showing that the Rapoport--Zink local model can fail to be flat.


Let $F/F_0$ be a ramified quadratic extension of discretely valued non-Archimedean fields with common residue field of characteristic not $2$; let us note that the residual characteristic $2$ case is fundamentally more difficult, and we do not omit it merely for simplicity.  Let $n\geq 2$ be an integer, and let
\[
   m := \lfloor n/2 \rfloor.
\]
Let $r + s = n$ be a partition of $n$; the pair $(r,s)$ is called the signature.  Let $I \subset \{0,\dotsc,m\}$ be a nonempty subset with the property that
\begin{equation}\label{disp:I_cond}
	n \text{ is even and } m-1 \in I \implies m \in I.
\end{equation}
Such subsets $I$ index the conjugacy classes of parahoric subgroups in quasi-split $GU_n(F/F_0)$; see \cite{paprap09}*{\s1.2.3}.  Attached to these data is the Rapoport--Zink local model $M^\naive_I$,%
\footnote{To be clear, $M_I^\naive$ depends on the signature $(r,s)$ as well as $I$, but we suppress the former in the notation.}
which has come to be called the ``naive'' local model since it is not flat in general.  See \s\ref{ss:naiveLM} for its explicit definition.  It is a projective scheme over $\Spec \O_E$, where the (local) reflex field $E := F$ if $r \neq s$ and $E:= F_0$ if $r = s$.  When $F$ is the $\QQ_p$-localization of an imaginary quadratic field $K$ in which $p$ ramifies, $M_I^\naive$ is a local model of a model over $\Spec\O_E$ of a $GU_n(K/\QQ)$-Shimura variety, as is explained for example in \cite{paprap09}*{\s1.5.4}.  See \s\ref{ss:Sh} for an example where we spell out such an integral Shimura model explicitly.

Let $M_I^\loc$ denote the (honest) local model, defined as the scheme-theoretic closure of the generic fiber in $M_I^\naive$.  As a first 
step towards a moduli characterization of $M_I^\loc$, Pappas proposed a new condition in \cite{pappas00} to add to the moduli problem defining $M_I^\naive$, called the \emph{wedge condition}; see \s\ref{ss:wedge_spin_conds}.  Denote by $M_I^\wedge$ the closed subscheme of $M_I^\naive$ cut out by the wedge condition.  In the maximal parahoric case $I = \{0\}$ (which is moreover a special maximal parahoric case when $n$ is odd), Pappas conjectured that $M_{\{0\}}^\wedge = M_{\{0\}}^\loc$, and he proved this conjecture in the case of signature $(n-1,1)$.  We will prove his conjecture in general in \cite{sm-decon}.

But for other $I$ the wedge condition is not enough.  The next advance came in \cite{paprap09} with Pappas and Rapoport's introduction of the \emph{spin condition}; see \s\ref{ss:wedge_spin_conds}.  Denote by $M_I^\spin$ the closed subscheme of $M_I^\wedge$ cut out by the spin condition.  Pappas and Rapoport conjectured that $M_I^\spin = M_I^\loc$, and it was shown in \cites{sm11d,sm14} that this equality at least holds on the level of topological spaces.  The starting point of the present paper is that the full equality of these schemes does \emph{not} hold in general.

\begin{ceg}\label{ceg}
For odd $n \geq 5$ and signature $(n-1,1)$, $M_{\{m\}}^\spin$ is not flat over $\Spec \O_E$.
\end{ceg}

See \s\ref{ss:failure}.  
We remark that 
the level structure in the counterexample is of special maximal parahoric type.

In response to the counterexample, in this paper we introduce a further refinement to the moduli problem defining $M_I^\naive$.  This defines a scheme $M_I$ which fits into a diagram of closed immersions
\[
   M_I^\loc \subset M_I \subset M_I^\spin \subset M_I^\wedge \subset M_I^\naive
\]
which are all equalities 
in the generic fiber; see \s\ref{ss:new_cond}.
In its formulation, our new condition is a close analog of the Pappas--Rapoport spin condition.  In its mathematical content, it gives a common refinement of the spin condition and the Kottwitz condition.  We conjecture that it solves the problem of characterizing $M_I^\loc$.

\begin{conj}\label{st:conj}
For any signature and nonempty $I$ satisfying \eqref{disp:I_cond}, $M_I$ is flat over $\Spec \O_E$, or in other words $M_I = M_I^\loc$.
\end{conj}

The main result of this paper is that $M_I$ at least corrects for Counterexample \ref{ceg}, i.e.\ we prove Conjecture \ref{st:conj} in the setting of the counterexample.

\begin{thm}\label{st:main_thm}
For odd $n$ and signature $(n-1,1)$, $M_{\{m\}} = M_{\{m\}}^\loc$.
\end{thm}



Our proof of the theorem 
is based on calculations of Arzdorf \cite{arzdorf09}, who 
studied in detail the local equations describing $M_{\{m\}}^\loc$ when $n$ is odd.  
In the setting of the theorem, Richarz observed that the local model is actually smooth \cite{arzdorf09}*{Prop.\ 4.16}.  The condition defining $M_{\{m\}}$ can be used to define a related formally smooth Rapoport--Zink space which plays an important role in the forthcoming paper \cite{rapoportsmithlingzhang?}.  Relatedly, Richarz's smoothness result and Theorem \ref{st:main_thm} also imply that a certain moduli problem of abelian schemes (which is an integral model for a unitary Shimura variety) is smooth; we make this explicit in \s\ref{ss:Sh}.

While the condition defining $M_I$ can be formulated for any $I$ and any signature, outside of Counterexample \ref{ceg}, we do not know the extent to which the inclusion $M_I \subset M_I^\spin$ fails to be an equality.  Indeed Pappas and Rapoport have obtained a good deal of computational evidence for the flatness of $M_I^\spin$ in low rank cases, and we do not know of any counterexamples to the flatness of $M_I^\spin$ when $n$ is even.

The organization of the paper is as follows.  In \s\ref{s:mod_prob} we review the definition of the naive, wedge, and spin local models, and we formulate our refined condition.  In \s\ref{s:special_case} we explain Counterexample \ref{ceg} and reduce the proof of Theorem \ref{st:main_thm} to Proposition \ref{st:X_1=0}, whose proof occupies \s\ref{s:proof}.  In \s\ref{s:remarks} we collect various remarks.  We show that the condition defining $M_I$ implies the Kottwitz condition in \s\ref{ss:kottwitz}, and in \s\ref{ss:wedge_power_analogs} we formulate some analogous conditions for other wedge powers.  These are closed conditions on $M_I^\naive$ which hold on the generic fiber, and therefore hold on $M_I^\loc$; we show that they imply the wedge condition.\footnote{It may also be interesting to note that, at least in the setting of Theorem \ref{st:main_thm}, the condition defining $M_I$ itself implies the wedge condition.  See Remark \ref{rk:newcond==>wedge}.}
In \s\ref{ss:Sh} we give the aforementioned application of Theorem \ref{st:main_thm} to an explicit integral model of a unitary Shimura variety. 
We conclude the paper in \s\ref{ss:PEL_setting} by explaining how to formulate these conditions in the general PEL setting, where they again imply the Kottwitz condition and automatically hold on the flat closure of the generic fiber in the naive local model.  To be clear, the conditions we formulate in \s\ref{ss:PEL_setting} will not suffice to characterize the flat closure in general, since for example they do not account for the spin condition in the ramified unitary setting.  But it would be interesting to see if they prove useful in other situations in which spin conditions do not arise.

\subsection*{Acknowledgements}  It is a pleasure to thank Michael Rapoport, whose inquiries about the spin condition in the setting of Counterexample \ref{ceg} led to the discovery of this counterexample, which in turn spawned the paper.  I also heartily thank him and George Pappas for a number of inspiring conversations related to this work.  I finally thank the referees for their helpful suggestions and remarks.

\subsection*{Notation}

Throughout the paper $F/F_0$ denotes a ramified quadratic extension of discretely valued, non-Archimedean fields with respective rings of integers $\O_F$ and $\O_{F_0}$, respective uniformizers $\pi$ and $\pi_0$ satisfying $\pi^2 = \pi_0$, and common residue field $k$ of characteristic not $2$.%
\footnote{Many of the papers we will refer to also assume that $k$ is perfect, but any facts we need for general $k$ will follow from the case of perfect $k$ by descent.} 
%

We work with respect to a fixed integer $n \geq 2$.  For $i \in \{1,\dotsc,n\}$, we write
\[
   i^\vee := n+1-i.
\]
For $i \in \{1,\dotsc,2n\}$, we write
\[
   i^* := 2n+1-i.
\]
For $S \subset \{1,\dotsc,2n\}$, we write
\[
   S^* = \{\,i^* \mid i \in S\,\} \quad\text{and}\quad S^\perp = \{1,\dotsc,2n\} \smallsetminus S^*.
\]
We also define
\[
   \Sigma S := \sum_{i\in S}i.
\]
For $a$ a real number, we write $\lfloor a \rfloor$ for the greatest integer $\leq a$, and $\lceil a \rceil$ for the least integer $\geq a$.
We write $a,\dotsc,\wh b,\dotsc,c$ for the list $a,\dotsc,c$ with $b$ omitted.

\section{The moduli problem}\label{s:mod_prob}
\numberwithin{equation}{subsection}

In this section we review the definition of $M_I^\naive$, $M_I^\wedge$, and $M_I^\spin$ from \cite{paprap09}, and we introduce our further refinement to the moduli problem.

\subsection{Linear-algebraic setup}\label{ss:setup}

Consider the vector space $F^n$ with its standard $F$-basis $e_1,\dotsc,e_n$.  Let
\[
   \phi\colon F^n \times F^n \to F
\]
denote the $F/F_0$-Hermitian form which is split with respect to the standard basis, i.e.
\begin{equation}\label{split}
   \phi(ae_i,be_j) = \ol a b \delta_{ij^\vee}, \quad a, b \in F,
\end{equation}
where $a \mapsto \ol a$ is the nontrivial element of $\Gal(F/F_0)$.  Attached to $\phi$ are the respective alternating and symmetric $F_0$-bilinear forms
\[
   F^n \times F^n \to F_0
\]
given by
\[
   \langle x,y \rangle := \frac 1 2 \tr_{F/F_0}\bigl( \pi^{-1}\phi(x,y) \bigr)
   \quad\text{and}\quad
   (x,y) := \frac 1 2 \tr_{F/F_0}\bigl( \phi(x,y) \bigr).
\]

For each integer $i = bn+c$ with $0 \leq c < n$, define the standard $\O_F$-lattice
\begin{equation}\label{Lambda_i}
   \Lambda_i := \sum_{j=1}^c\pi^{-b-1}\O_F e_j + \sum_{j=c+1}^{n} \pi^{-b}\O_F e_j \subset F^n.
\end{equation}
For all $i$, the \aform-dual of $\Lambda_i$ in $F^n$ is $\Lambda_{-i}$, by which we mean that
\[
   \bigl\{\,x\in F^n \bigm| \langle\Lambda_i,x\rangle \subset \O_{F_0} \,\bigr\} = \Lambda_{-i},
\]
and
\begin{equation}\label{pairing}
   \Lambda_i \times \Lambda_{-i} \xra{\aform} \O_{F_0}
\end{equation}
is a perfect $\O_{F_0}$-bilinear pairing.  Similarly, $\Lambda_{n-i}$ is the \sform-dual of $\Lambda_i$ in $F^n$.
The $\Lambda_i$'s form a complete, periodic, self-dual lattice chain
\[
   \dotsb \subset \Lambda_{-2} \subset \Lambda_{-1} \subset \Lambda_0 \subset \Lambda_1 \subset \Lambda_2 \subset \dotsb.
\]

\subsection{Naive local model}\label{ss:naiveLM}

Let $I \subset \{0,\dotsc,m\}$ be a nonempty subset satisfying \eqref{disp:I_cond}, and let $r + s = n$ be a partition.  As in the introduction, let
\[
   E = F \quad\text{if}\quad r \neq s \quad\text{and}\quad E = F_0 \quad\text{if}\quad r = s.
\]
The \emph{naive local model $M_I^\naive$} is a projective scheme over $\Spec \O_E$.  It represents the moduli problem that sends each $\O_E$-algebra $R$ to the set of all families
\[
   (\F_i \subset \Lambda_i \otimes_{\O_{F_0}}R)_{i\in \pm I + n\ZZ}
\]
such that
\begin{enumerate}
\renewcommand{\theenumi}{LM\arabic{enumi}}
\item\label{it:LM1}
for all $i$, $\F_i$ is an $\O_F \otimes_{\O_{F_0}} R$-submodule of $\Lambda_i \otimes_{\O_{F_0}} R$, and an $R$-direct summand of rank $n$;
\item\label{it:LM2}
for all $i < j$, the natural arrow $\Lambda_i \otimes_{\O_{F_0}} R \to \Lambda_j \otimes_{\O_{F_0}} R$ carries $\F_i$ into $\F_j$;
\item\label{it:LM3}
for all $i$, the isomorphism $\Lambda_i \otimes_{\O_{F_0}} R \xra[\undertilde]{\pi \otimes 1} \Lambda_{i-n} \otimes_{\O_{F_0}} R$ identifies
\[
   \F_i \isoarrow \F_{i-n};
\]
\item\label{it:LM4}
for all $i$, the perfect $R$-bilinear pairing
\[
   (\Lambda_i \otimes_{\O_{F_0}} R) \times (\Lambda_{-i} \otimes_{\O_{F_0}} R)
   \xra{\aform \otimes R} R
\]
identifies $\F_i^\perp$ with $\F_{-i}$ inside $\Lambda_{-i} \otimes_{\O_{F_0}} R$; and
\item\label{it:kott_cond} (Kottwitz condition) for all $i$, the element $\pi \otimes 1 \in \O_F \otimes_{\O_{F_0}} R$ acts on $\F_i$ as an $R$-linear endomorphism with characteristic polynomial
\[
   \det(T\cdot \id - \pi \otimes 1 \mid \F_i) = (T-\pi)^s(T+\pi)^r \in R[T].
\]
\setcounter{filler}{\value{enumi}}
\end{enumerate}

When $r = s$, the polynomial on the right-hand side in the Kottwitz condition is to be interpreted as $(T^2 - \pi_0)^s$, which makes sense over any $\O_{F_0}$-algebra.

\subsection{Wedge and spin conditions}\label{ss:wedge_spin_conds}


We continue with $I$ and $(r,s)$ as before.  The \emph{wedge condition}
on an $R$-point $(\F_i)_i$ of $M_I^\naive$ is that
\begin{enumerate}
\renewcommand{\theenumi}{LM\arabic{enumi}}
\setcounter{enumi}{\value{filler}}
\item\label{it:wedge_cond}
if $r \neq s$, then for all $i$,
\[
   \sideset{}{_R^{s+1}}{\bigwedge} (\,\pi\otimes 1 + 1 \otimes \pi \mid \F_i\,) = 0
   \quad\text{and}\quad
   \sideset{}{_R^{r+1}}{\bigwedge} (\,\pi\otimes 1 - 1 \otimes \pi \mid \F_i\,) = 0.
\]
(There is no condition when $r=s$.)
\setcounter{filler}{\value{enumi}}
\end{enumerate}
The \emph{wedge local model $M_I^\wedge$} is the closed subscheme in $M_I^\naive$ 
where the wedge condition is satisfied.

We next turn to the spin condition, which involves the symmetric form \sform and requires some more notation.  Let
\[
   V := F^n \otimes_{F_0} F,
\]
regarded as an $F$-vector space of dimension $2n$ via the action of $F$ on the right tensor factor.  Let
\[
   W := \sideset{}{_F^n}\bigwedge V.
\]  
When $n$ is even, \sform is split over $F^n$, by which we mean that there is an $F_0$-basis $f_1,\dotsc,f_{2n}$ such that $(f_i,f_j) = \delta_{ij^*}$.  In all cases, $\sform \otimes_{F_0} F$ is split over $V$.
Hence there is a canonical decomposition
\[
   W = W_1 \oplus W_{-1}
\]
of $W$ as an $SO\bigl(\sform\bigr)(F) \isom SO_{2n}(F)$-representation.

Intrinsically, $W_1$ and $W_{-1}$ have the property that for any totally isotropic $n$-dimensional subspace $\F \subset V$, the line $\bigwedge_F^n\F \subset W$ is contained in $W_1$ or in $W_{-1}$, and in this way they distinguish the two connected components of the orthogonal Grassmannian $\OGr(n,V)$ over $\Spec F$.  Concretely, $W_1$ and $W_{-1}$ can be described as follows.  Let $f_1,\dotsc,f_{2n}$ be an $F$-basis for $V$.  For $S = \{i_1< \dots < i_n\} \subset \{1,\dotsc,2n\}$ of cardinality $n$, let
\begin{equation}\label{disp:f_S}
   f_S := f_{i_1} \wedge \dotsb \wedge f_{i_n} \in W,
\end{equation}
and let $\sigma_S$ be the permutation on $\{1,\dotsc,2n\}$ sending
\[
   \{1,\dotsc,n\} \xra[\undertilde]{\sigma_S} S
\] in increasing order and
\[
   \{n+1,\dotsc,2n\} \xra[\undertilde]{\sigma_S} \{1,\dotsc,2n\} \smallsetminus S
\]
in increasing order.  For varying $S$ of cardinality $n$, the $f_S$'s form a basis of $W$, and we define an $F$-linear operator $a$ on $W$ by defining it on them:
\[
   a(f_S) := \sgn(\sigma_S)f_{S^\perp}.
\]
Then, when $f_1,\dotsc,f_{2n}$ is a split basis for \sform,
\begin{equation}\label{disp:W_+-1}
   W_{\pm 1} = \spann_F\bigl\{\,f_S \pm \sgn(\sigma_S)f_{S^\perp}\bigm| \#S=n\,\bigr\}
\end{equation}
is the $\pm 1$-eigenspace for $a$.  Any other split basis is carried onto $f_1,\dotsc,f_{2n}$ by an element $g$ in the orthogonal group.  If $\det g = 1$ then $W_1$ and $W_{-1}$ are both $g$-stable, whereas if $\det g = -1$ then $W_1$ and $W_{-1}$ are interchanged by $g$.  In this way $W_1$ and $W_{-1}$ are independent of choices up to labeling.

For the rest of the paper, we pin down a particular choice of $W_1$ and $W_{-1}$ as in \cite{paprap09}*{\s7.2}.  If $n = 2m$ is even, then
\[
   -\pi^{-1}e_1,\dotsc,-\pi^{-1}e_m, e_{m+1},\dotsc,e_n,e_1,\dotsc,e_m, \pi e_{m+1},\dotsc, \pi e_{n}
\]
is a split ordered $F_0$-basis for \sform in $F^n$, and we take $f_1,\dotsc,f_{2n}$ to be the image of this basis in $V$.  If $n = 2m+1$ is odd, then we take $f_1,\dotsc,f_{2n}$ to be the split ordered basis
\begin{equation}\label{disp:PR_basis}
\begin{gathered}
   -\pi^{-1} e_1 \otimes 1, \dotsc, -\pi^{-1} e_m \otimes 1, e_{m+1}\otimes 1 - \pi e_{m+1} \otimes \pi^{-1},\\ e_{m+2} \otimes 1,\dotsc, e_n\otimes 1, e_1\otimes 1,\dotsc, e_m \otimes 1,\\ \frac{e_{m+1} \otimes 1 + \pi e_{m+1} \otimes \pi^{-1}}2, \pi e_{m+2} \otimes 1, \dotsc, \pi e_{n} \otimes 1.
\end{gathered}
\end{equation}
These are the same choices that are used in \citelist{\cite{sm11d}\cite{sm14}}.  For $\Lambda$ an $\O_{F_0}$-lattice in $F^n$,
\[
   W(\Lambda) := \sideset{}{_{\O_F}^n}\bigwedge (\Lambda \otimes_{\O_{F_0}} \O_F)
\]
is naturally an $\O_F$-lattice in $W$, and we define
\[
   W(\Lambda)_{\pm 1} := W_{\pm 1} \cap W(\Lambda).
\]


We now formulate the spin condition.  If $R$ is an $\O_F$-algebra, then the \emph{spin condition} on an $R$-point $(\F_i \subset \Lambda_i \otimes_{\O_{F_0}}R )_i$ of $M_I^\naive$ is that
\begin{enumerate}
\renewcommand{\theenumi}{LM\arabic{enumi}}
\setcounter{enumi}{\value{filler}}
\item\label{it:spin_cond}
for all $i$, the line $\bigwedge_R^n \F_i \subset W(\Lambda_i) \otimes_{\O_{F}}R$ is contained in
\[
   \im\bigl[ W(\Lambda_i)_{(-1)^s} \otimes_{\O_{F}} R 
      \to W (\Lambda_i) \otimes_{\O_{F}}R
   \bigr].
\]
\setcounter{filler}{\value{enumi}}
\end{enumerate}
This defines the spin condition when $r \neq s$.  When $r = s$, $W_{\pm 1}$ is defined over $F_0$ since \sform is already split before extending scalars $F_0 \to F$, and the spin condition on $M_{I,\O_F}^\naive$ descends to $M_I^\naive$ over $\Spec \O_{F_0}$.  In all cases, the \emph{spin local model $M_I^\spin$} is the closed subscheme of $M_I^\wedge$
where the spin condition is satisfied.

\begin{rk}
Our definition of $a$ above is the same as in \citelist{\cite{sm11d}*{\s2.5}\cite{sm14}*{\s2.4}}.  As noted in those papers, this agrees only up to sign with the analogous operators denoted $a_{f_1\wedge \dotsb \wedge f_{2n}}$ in \cite{paprap09}*{disp.\ 7.6} and $a$ in \cite{sm11b}*{\s2.3}, and there is a sign error in the statement of the spin condition in \cite{paprap09}*{\s7.2.1} tracing to this discrepancy.
\end{rk}

\subsection{Interlude: the sign of $\sigma_S$}

Here is an efficient means to calculate the sign $\sgn(\sigma_S)$ occurring in the expression \eqref{disp:W_+-1} for $W_{\pm 1}$.

\begin{lem}\label{st:sign_sigma_S}
For $S \subset \{1,\dotsc,2n\}$ of cardinality $n$,
\[
   \sgn(\sigma_S) = (-1)^{ \Sigma S + \lceil n/2 \rceil}.
\]
\end{lem}

\begin{proof}
Let $P$ denote the permutation matrix attached to $\sigma_S$, so that the $(i,j)$-entry of $P$ is $\delta_{i,\sigma_S(j)}$.  We compute $\det P$.  Say $S = \{i_1 < \dotsb < i_n\}$.  Using Laplace expansion along the $n$th column of $P$, then along the $(n-1)$st column, then along\dots, then along the first column, we find
\begin{align*}
   \det P &= (-1)^{i_n + n}(-1)^{i_{n-1} + n-1}\dotsm (-1)^{i_1 + 1}\det I_n\\
          &= (-1)^{\Sigma S + \frac{n(n+1)}2}\\
          &= (-1)^{\Sigma S + \lceil n/2 \rceil}. \qedhere
\end{align*}
\end{proof}

\subsection{Further refinement}\label{ss:new_cond}

We now formulate our refinement to the moduli problem defining $M_I^\spin$.  The idea is to further restrict the intersection in the definition of $W(\Lambda_i)_{\pm 1}$ in a way that incorporates a version of the Kottwitz condition.  We continue with the notation from before.

The operator $\pi \otimes 1$ acts $F$-linearly and semisimply on $V$ with eigenvalues $\pi$ and $-\pi$.  Let
\[
   V_\pi \quad\text{and}\quad V_{-\pi}
\]
denote its respective eigenspaces.  Let
\[
   W^{r,s} := \sideset{}{_F^r}\bigwedge V_{-\pi} \otimes_F \sideset{}{_F^s}\bigwedge V_\pi.
\]
Then $W^{r,s}$ is naturally a subspace of $W$, and
\[
   W = \bigoplus_{r+s=n}W^{r,s}.
\]
Let
\[
   W_{\pm 1}^{r,s} := W^{r,s} \cap W_{\pm 1}.
\]
For any $\O_{F_0}$-lattice $\Lambda$ in $F^n$, let
\[
   W(\Lambda)_{\pm 1}^{r,s} := W_{\pm 1}^{r,s} \cap W(\Lambda) \subset W.
\]

Our new condition is just the analog of the spin condition with $W(\Lambda_i)_{(-1)^s}^{r,s}$ in place of $W(\Lambda_i)_{(-1)^s}$.  To lighten notation, for $R$ an $\O_F$-algebra, define
\[
   L_i^{r,s}(R) := \im\bigl[ W(\Lambda_i)_{(-1)^s}^{r,s} \otimes_{\O_F} R \to W(\Lambda_i) \otimes_{\O_{F}}R\bigr].
\]
The condition on an $R$-point $(\F_i \subset \Lambda_i \otimes_{\O_{F_0}} R)_i$ of $M_I^\naive$ is that
\begin{enumerate}
\renewcommand{\theenumi}{LM\arabic{enumi}}
\setcounter{enumi}{\value{filler}}
\item\label{it:new_cond}
for all $i$, the line $\bigwedge_R^n \F_i \subset W(\Lambda_i)\otimes_{\O_{F}}R$ is contained in $L_i^{r,s}(R)$.
\end{enumerate}
This defines the condition when $r \neq s$.  When $r = s$, the subspaces $W^{r,s}$ and $W_{\pm 1}$ are Galois-stable, and the condition descends from $M_{I,\O_F}^\naive$ to $M_I^\naive$.  In all cases, we write $M_I$ for the locus in $M_I^\spin$ where the condition is satisfied.

Clearly $M_I$ is a closed subscheme of $M_I^\spin$.  
Our new condition is also already satisfied in the generic fiber:  if $R$ is an $F$-algebra and $(\F \subset V \otimes_F R)$ is an $R$-point on $M_I^\naive$, then by the Kottwitz condition it is automatic that $\bigwedge_R^n \F \subset W^{r,s} \otimes_F R$ inside $W \otimes_F R$.  Thus there is a diagram of closed immersions
\[
   M_I^\loc \subset M_I \subset M_I^\spin \subset M_I^\wedge \subset M_I^\naive
\]
which are all equalities between generic fibers.  Conjecture \ref{st:conj} is that $M_I^\loc = M_I$ in general.

\section{The case of $n$ odd, $I = \{m\}$, and signature $(n-1,1)$}\label{s:special_case}

The main goal of this section and \s\ref{s:proof} is to prove Theorem \ref{st:main_thm}.  We explain Counterexample \ref{ceg} along the way.  Throughout we specialize to the case that $n = 2m + 1$ is odd, $I =\{m\}$, and $(r,s) = (n-1,1)$.  This is a case of special maximal parahoric level structure studied in detail by Arzdorf in \cite{arzdorf09}, and we begin by reviewing the calculations of his that will be relevant for us.  Of course $\O_E = \O_F$.  To lighten notation, we suppress the set $\{m\}$, so that $M^\naive = M_{\{m\}}^\naive$, $M = M_{\{m\}}$, etc.

\subsection{Review of Arzdorf's calculations}\label{ss:arzdorf_calcs}

It is clear from its definition that $M^\naive$ is naturally a closed subscheme of the Grassmannian $\Gr(n,\Lambda_m \otimes_{\O_{F_0}} \O_F)$ over $\Spec \O_F$.  Arzdorf computes an affine chart on the special fiber $M_{k}^\naive$ around its ``worst point'' in  \cite{arzdorf09}*{\s4} as the restriction of one of the standard open affine charts on the Grassmannian.%
\footnote{Strictly speaking, Arzdorf obtains equations describing the $\ol k$-points in an open subscheme of $M^\naive$, but implicit in his discussion are equations defining this subscheme itself, once one additionally incorporates the Kottwitz condition.}
Here the ``worst point'' is the $k$-point
\[
   (\pi\otimes 1) \cdot (\Lambda_m \otimes_{\O_{F_0}} k) \subset \Lambda_m \otimes_{\O_{F_0}} k.
\]
The reason for this terminology is that the geometric special fiber $M_{\ol k}^\naive$ embeds into an affine flag variety for $GU_n$, where it decomposes (as a topological space) into a disjoint union of Schubert cells, and the (image of the) worst point is the unique closed Schubert cell.  See \cite{paprap09}*{\s2.4.2, \s5.5}.

Following Arzdorf, take the ordered $k$-basis
\begin{multline}\label{disp:Arzdorf_basis}
   e_{m+2} \otimes 1, \dotsc, e_n\otimes 1, \pi^{-1}e_1\otimes 1,\dotsc, \pi^{-1}e_m\otimes 1, e_{m+1}\otimes 1,\\ \pi e_{m+2} \otimes 1, \dotsc, \pi e_n \otimes 1, e_1\otimes 1,\dotsc,e_m\otimes 1, \pi e_{m+1} \otimes 1
\end{multline}
for $\Lambda_m \otimes_{\O_{F_0}} k$.  With respect to this basis, the standard open affine chart $U^{\Gr}$ on $\Gr(n,\Lambda_m \otimes_{\O_{F_0}}k)$ containing the worst point is the $k$-scheme of $2n\times n$-matrices
\begin{equation}\label{disp:U^Gr}
\begin{pmatrix}
   X\\
   I_n
\end{pmatrix}
\end{equation}
(the worst point itself corresponds to $X = 0$).
Define
\begin{equation}\label{U's}
   U^\loc \subset U \subset U^\spin \subset U^\wedge \subset U^\naive \subset U^{\Gr}
\end{equation}
by intersecting $U^{\Gr}$ with $M^\loc$, $M$, $M^\spin$, $M^\wedge$, and $M^\naive$, respectively. Write
\begin{equation}
   X = \begin{pmatrix}
          X_1 & X_2\\
          X_3 & X_4
       \end{pmatrix},
\end{equation}
where $X_1$ is of size $(n-1)\times (n-1)$, $X_2$ is of size $(n-1)\times 1$, $X_3$ is of size $1 \times (n-1)$, and $X_4$ is scalar.

Arzdorf shows that $X_4^2 = 0$ on $U^\naive$ \cite{arzdorf09}*{p.\ 701}.  Let $U_{X_4 = 0}^\naive$ be the closed subscheme of $U^\naive$ defined by imposing $X_4 = 0$.  Arzdorf shows that $X_2 = 0$ on $U_{X_4=0}^\naive$, and that conditions \eqref{it:LM1}--\eqref{it:LM4} translate to%
\footnote{Strictly speaking, Arzdorf does not explicitly address the equations arising from the lattice inclusion $\Lambda_{m+1} \subset \pi^{-1}\Lambda_m$ in \eqref{it:LM2}.  It is straightforward to verify that these equations add nothing further to \eqref{disp:naive_eqns} when $X_4 = 0$.}
\begin{equation}\label{disp:naive_eqns}
   -J X_3^\Rt X_3 = X_1 + J X_1^\Rt J, \quad X_1^2 = 0, \quad\text{and}\quad X_3 X_1 = 0,
\end{equation}
where $J$ is the $(n-1) \times (n-1)$-matrix
\[
   J := \begin{pmatrix}
       & & & & & 1\\
       & & & & \iddots\\
       & & & 1\\
       & &-1\\
       & \iddots\\
       -1
       \end{pmatrix}
       \qquad\text{($m \times m$ blocks)}.
\]
Arzdorf does not translate the Kottwitz condition, since it is automatically satisfied on the reduced special fiber of $M^\naive$.  Let us do so now.  Regarding the columns of \eqref{disp:U^Gr} as a basis for the subspace $\F_m$, the operator $\pi\otimes 1$ acts as
\[
   (\pi \otimes 1) \cdot \begin{pmatrix} X\\ I_n \end{pmatrix}
     = \begin{pmatrix} 0 \\ X \end{pmatrix}.
\]
It follows that the Kottwitz condition on $U^{\Gr}$ is that%
\footnote{It is an easy consequence of \eqref{it:LM3} and \eqref{it:LM4} that the Kottwitz condition for $\F_m$ implies the Kottwitz condition for $\F_{m+1}$.}
\[
   \charac_X(T) = T^n.
\]
When the rightmost column of $X$ is zero, which is the case on $U_{X_4=0}^\naive$, the Kottwitz condition becomes
\begin{equation}\label{disp:Kottwitz_cond}
	\charac_{X_1}(T) = T^{n-1}.
\end{equation}
We conclude that $U_{X_4 = 0}^\naive$ is the $k$-scheme of $n\times(n-1)$-matrices $\bigl(\begin{smallmatrix} X_1 \\ X_3\end{smallmatrix}\bigr)$ satisfying \eqref{disp:naive_eqns} and \eqref{disp:Kottwitz_cond}.

In our case of signature $(n-1,1)$, Arzdorf shows in \cite{arzdorf09}*{\s4.5} that the wedge condition on $U_{X_4=0}^\naive$ is the condition
\begin{equation}\label{disp:wedge_eqns}
   \sideset{}{^2}\bigwedge \begin{pmatrix} X_1 \\ X_3 \end{pmatrix} = 0.
\end{equation}
When \eqref{disp:wedge_eqns} is satisfied, the Kottwitz condition \eqref{disp:Kottwitz_cond} reduces to the condition
\begin{equation}\label{disp:Tr_cond}
   \tr X_1 = 0.
\end{equation}
Arzdorf proves that for any signature, $U^\loc$ is reduced \cite{arzdorf09}*{Th.\ 2.1}, so that
\[
   U^\loc \subset U_{X_4=0}^\naive.
\]
It is an observation of Richarz \cite{arzdorf09}*{Prop.\ 4.16} that for signature $(n-1,1)$, the map
\begin{equation}\label{disp:U_map}
   \begin{pmatrix} X_1 \\ X_3 \end{pmatrix} \mapsto X_3
\end{equation}
induces an isomorphism
\begin{equation}\label{disp:U^loc_isom}
   U^\loc \isoarrow \AA_k^{n-1}.
\end{equation}

\subsection{Failure of flatness of $M^\spin$}\label{ss:failure}

In this subsection we explain Counterexample \ref{ceg}.  We use the calculations of the previous subsection, whose notation we retain.  The first item of business is the following.

\begin{lem}\label{st:spin=>X_4=0}
On $U^\naive$, the spin condition implies that $X_4 = 0$.
\end{lem}

Before proving the lemma, we need some more notation.  Consider the ordered $\O_{F_0}$-basis
\[
   \pi^{-1}e_1,\dotsc, \pi^{-1}e_m, e_{m+1},\dotsc,e_n, e_1,\dotsc, e_m, \pi e_{m+1}, \dotsc, \pi e_n
\]
for $\Lambda_m$.  After extending scalars $\O_{F_0} \to k$, this gives Arzdorf's basis \eqref{disp:Arzdorf_basis}, but in a different order.  After extending scalars $\O_{F_0} \to F$, this gives an ordered $F$-basis for $V$, and for $S \subset \{1,\dotsc,2n\}$ a subset of cardinality $n$, we define
\begin{equation}\label{disp:e_S}
	e_S \in W
\end{equation}
with respect to this basis for $V$ as in \eqref{disp:f_S}.  For varying $S$ of cardinality $n$, the $e_S$'s form an $\O_F$-basis for $W(\Lambda_m) \subset W$.  Given an $\O_F$-algebra $R$, we will often abuse notation and continue to write $e_S$ for its image in $W(\Lambda_m) \otimes_{\O_{F}} R$.

Let $f_1,\dotsc,f_{2n}$ denote the split ordered basis \eqref{disp:PR_basis} for $V$.

Now let $R$ be a $k$-algebra, and let
\[
   N_{\pm 1} := \im\bigl[ W(\Lambda_m)_{\pm 1} \otimes_{\O_{F}} R 
      \to W(\Lambda_m) \otimes_{\O_{F}}R
   \bigr].
\]

\begin{proof}[Proof of Lemma \ref{st:spin=>X_4=0}]
Let $(\F_m \subset \Lambda_m \otimes_{\O_{F_0}}R)$ be an $R$-point on $U^\naive$.
Regarding each column of the matrix \eqref{disp:U^Gr} as an $R$-linear combination of Arzdorf's basis elements \eqref{disp:Arzdorf_basis}, the wedge of these columns (say from left to right) is expressible as a linear combination
\begin{equation}\label{disp:lin_comb}
   \sum_{\substack{S \subset \{1,\dotsc,2n\}\\ \#S = n}} c_S e_S \in W(\Lambda_m) \otimes_{\O_{F}} R, \quad c_S \in R.
\end{equation}
We have
\[
   c_{\{m+1,n+1,\dotsc,\wh{n+m+1},\dotsc,2n\}} = (-1)^m X_4.
\]
The spin condition for arbitrary signature $(r,s)$ is that the element \eqref{disp:lin_comb} is contained in $N_{(-1)^s}$.  We are going to show that if \eqref{disp:lin_comb} is contained in $N_1$ or in $N_{-1}$, then $c_{\{m+1,n+1,\dotsc,\wh{n+m+1},\dotsc,2n\}}$ must vanish.  Let $\varepsilon \in \{\pm 1\}$.

By \eqref{disp:W_+-1}, every element in  $W(\Lambda_m)_{\varepsilon}$ is an $F$-linear combination of elements of the form
\begin{equation}\label{disp:f_Spmf_Sperp}
   f_S + \varepsilon \sgn(\sigma_S) f_{S^\perp}.
\end{equation}
With respect to the $e_S$-basis for $W$, the only elements of the form \eqref{disp:f_Spmf_Sperp} which can possibly involve $e_{\{m+1,n+1,\dotsc,\wh{n+m+1},\dotsc,2n\}}$ are for $S$ one of the sets
\[
   \{m+1,n+1,\dotsc,\wh{n+m+1},\dotsc,2n\} \quad\text{and}\quad \{n+1,\dotsc,2n\}.
\]
These are also the only two sets for which \eqref{disp:f_Spmf_Sperp} can possibly involve $e_{\{n+1,\dotsc,2n\}}$.  Both of these sets are self-perp.  For one of them, which we call $T$, \eqref{disp:f_Spmf_Sperp} equals $2 f_T$, and for the other, \eqref{disp:f_Spmf_Sperp} equals $0$, as follows for example from Lemma \ref{st:sign_sigma_S}.  The element $2 f_T$ is of the form
\begin{multline*}
   (\text{unit in }\O_F^\times)\cdot (e_{m+1} \otimes 1 \pm \pi e_{m+1} \otimes \pi^{-1})\\ \wedge (e_1\otimes 1) \wedge \dotsb \wedge (e_m\otimes 1) \wedge (\pi e_{m+2} \otimes 1) \wedge \dotsb \wedge (\pi e_n \otimes 1),
\end{multline*}
which in turn is of the form
\begin{equation}\label{disp:2f_T}
   (\text{unit in }\O_F^\times) \cdot e_{\{m+1,n+1,\dotsc,\wh{n+m+1},\dotsc,2n\}} \pm \pi^{-1}(\text{unit in }\O_F^\times)  \cdot e_{\{n+1,\dotsc,2n\}}.
\end{equation}

Now suppose we have an $F$-linear combination
\begin{equation}\label{disp:lin_comb2}
   c_T \cdot 2f_T + \sum_{S\neq T}c_S\bigl(f_S + \varepsilon \sgn(\sigma_S) f_{S^\perp}\bigr) \in W(\Lambda_m)_{\varepsilon}.
\end{equation}
By writing this as a linear combination of $e_S$'s and looking at the $e_{\{n+1,\dotsc,2n\}}$-term, we conclude from the above discussion and \eqref{disp:2f_T} that
\[
   \ord_\pi (c_T) \geq 1.
\]
This and \eqref{disp:2f_T} imply that the $e_{\{m+1,n+1,\dotsc,\wh{n+m+1},\dotsc,2n\}}$-term in \eqref{disp:lin_comb2} dies in $N_\varepsilon$, since $\pi R = 0$.
\end{proof}

By the lemma $U^\spin \subset U_{X_4=0}^\naive$, and we conclude that $U^\spin$ is the $k$-scheme of matrices satisfying \eqref{disp:naive_eqns}, \eqref{disp:wedge_eqns}, and \eqref{disp:Tr_cond}, plus the spin condition.  When $n \geq 5$, we are going to show that $M^\spin$ is not flat over $\Spec \O_F$ by showing that $U^\spin \neq U^\loc$.  The remaining calculation that we need is the following.  We continue with our $k$-algebra $R$ and $R$-modules $N_{\pm 1}$.

\begin{lem}\label{st:N_lem}
For $\varepsilon \in \{\pm 1\}$, the elements
\[
   e_{\{n+1,\dotsc,2n\}} \quad\text{and}\quad e_{\{i,n+1,\dotsc,\wh{n+i},\dotsc,2n\}} \quad\text{for}\quad i \in \{1,\dotsc,\wh{m+1},\dotsc,n\}
\] are contained in $N_\varepsilon$.
\end{lem}

\begin{proof}
In the notation of proof of Lemma \ref{st:spin=>X_4=0}, the image of $\pi f_T$ in $N_\varepsilon$ is of the form
\[
   (\text{unit})\cdot e_{\{n+1,\dotsc,2n\}},
\]
which takes care of the first element on our list.

For $1 \leq i \leq m$, let
\[
   S_1 := \{i,m+1,n+1,\dotsc,2n\}\smallsetminus\{n+i,n+m+1\}
\]
and
\[
   S_2 := \{i,n+1,\dotsc,2n\} \smallsetminus \{n+i\}.
\]
Then
\[
   S_1^\perp = \{m+1,i^\vee,1,\dotsc,2n\} \smallsetminus \{n+m+1,n+i^\vee\}
\]
and
\[
   S_2^\perp = \{i^\vee,1\dotsc,2n\} \smallsetminus \{n+i^\vee\}.
\]
By Lemma \ref{st:sign_sigma_S},
\[
   \sgn(\sigma_{S_1}) = -1 \quad\text{and}\quad \sgn(\sigma_{S_2}) = 1.
\]
Hence
\[
\begin{split}
   \pi(f_{S_1} - \varepsilon f_{S_1^\perp}) &= (- \pi^{-1} e_i \otimes 1) \wedge (e_{m+1} \otimes \pi - \pi e_{m+1} \otimes 1)\\
   &\qquad\qquad \wedge f_{n+1} \wedge \dotsb \wedge \wh f_{n+i} \wedge \dotsb \wedge \wh f_{n+m+1} \wedge \dotsb \wedge f_{2n}\\
   &\qquad -\varepsilon (e_{m+1} \otimes \pi - \pi e_{m+1} \otimes 1)\\
   &\qquad\qquad\qquad \wedge f_{i^\vee} \wedge f_{n+1} \wedge \dotsb \wedge \wh f_{n+m+1} \wedge \dotsb \wedge \wh f_{n+i^\vee} \wedge \dotsb \wedge f_{2n}\\
   &\in W(\Lambda_m)_\varepsilon.
\end{split}
\]
This equals
\[
   - \pi e_{S_1} + (-1)^{m-1}e_{S_2} - \varepsilon \pi e_{S_1^\perp} + \varepsilon (-1)^{m+1} e_{S_2^\perp},
\]
which, since $\pi R = 0$, has image
\begin{equation}\label{disp:e_comb1}
	(-1)^{m-1}(e_{S_2} + \varepsilon e_{S_2^\perp})
\end{equation}
in $N_\varepsilon$.  Similarly,
\[
\begin{split}
   \pi(f_{S_2} + \varepsilon f_{S_2^\perp}) = {}&(- \pi^{-1} e_i \otimes 1) \wedge f_{n+1} \wedge \dotsb \wedge \wh f_{n+i} \wedge \dotsb \wedge f_{n+m}\\ 
   &\qquad\qquad\wedge \frac{e_{m+1} \otimes \pi + \pi e_{m+1} \otimes 1}2 \wedge f_{n+m+2} \wedge \dotsb \wedge f_{2n}\\
   &\qquad +\varepsilon f_{i^\vee} \wedge f_{n+1} \wedge \dotsb \wedge f_{n+m}\\
   &\qquad\qquad\qquad \wedge  \frac{e_{m+1} \otimes \pi + \pi e_{m+1} \otimes 1}2\\
   &\qquad\qquad\qquad\qquad\wedge f_{n+m+2} \wedge \dotsb \wh f_{n+i^\vee} \wedge \dotsb \wedge f_{2n}
\end{split}
\]
is in $W(\Lambda_m)_\varepsilon$ and has image
\begin{equation}\label{disp:e_comb2}
   -\frac 1 2 (-e_{S_2} + \varepsilon e_{S_2^\perp})
\end{equation}
in $N_\varepsilon$.  Plainly $e_{S_2}$ and $e_{S_2^\perp}$ are in the $k$-span of \eqref{disp:e_comb1} and \eqref{disp:e_comb2}, which proves the rest of the lemma.
\end{proof}

Now let $R$ be a $k$-algebra with a nonzero element $x$ such that $x^2 = 0$.  Suppose that $n \geq 5$, and take
\[
   X_1 = \diag(x,-x,0,\dotsc,0,-x,x) \quad\text{and}\quad X_3 = 0.
\]
Then $X_1$ and $X_3$ satisfy \eqref{disp:naive_eqns}, \eqref{disp:wedge_eqns}, and \eqref{disp:Tr_cond}.  Plugging them into \eqref{disp:U^Gr}, and taking $X_2$ and $X_4$ both $0$, the wedge of the columns of \eqref{disp:U^Gr} translates to a linear combination of the basis elements
\[
   e_{\{n+1,\dotsc,2n\}} \quad\text{and}\quad
   e_{\{i,n+1,\dotsc,\wh{n+i},\dotsc,2n\}} \quad\text{for}\quad i \in \{m-1,m,m+2,m+3\}.
\]
By Lemma \ref{st:N_lem}, such a linear combination is contained in $N_\varepsilon$ for any $\varepsilon$.
We conclude that $X_1$ and $X_3$ determine an $R$-point on $U^\spin$.  But in the case of signature $(n-1,1)$, any point on $U^\loc$ with $X_3 = 0$ must have $X_1 = 0$, via \eqref{disp:U^loc_isom}.  This exhibits that $U^\spin \neq U^\loc$.

\subsection{Flatness of $M$}\label{ss:reduction}

In this subsection we reduce Theorem \ref{st:main_thm}, which asserts that the scheme $M$ is flat over $\Spec \O_F$, to Proposition \ref{st:X_1=0} below, which we will subsequently prove in \s\ref{s:proof}.

Our goal is to show that the closed immersion $M^\loc \subset M$ is an equality.  Since this is an equality
between generic fibers
and $M^\loc$ is flat over $\O_F$, it suffices to show that it is also an equality 
between special fibers.  
For this we may assume that $k$ is algebraically closed.  

Consider the closed immersions
\[
   M_k^\loc \subset M_k \subset M_k^\spin \subset M_k^\wedge.
\]
As discussed at the beginning of \s\ref{ss:arzdorf_calcs}, these schemes all embed into an affine flag variety, where they topologically decompose into a union of Schubert cells.  By Richarz's result \cite{arzdorf09}*{Prop.\ 4.16} in the situation at hand or by \cite{sm11d}*{Main Th.}\ in general for odd $n$, the Schubert cells occurring in them are all the same.  In the present situation, there are just two Schubert cells that occur, the ``worst point'' and its complement $C$ in these schemes, as follows from \cite{sm11d}*{Cor.\ 5.6.2} and the calculation of the relevant admissible set in \cite{paprap09}*{\s2.4.2}.  Arzdorf shows that $M_k^\wedge$ contains an open reduced subscheme in \cite{arzdorf09}*{Prop.\ 3.2}.  Therefore the entire open cell $C$ must be reduced in $M_k^\wedge$, which implies the same for all of the schemes in the display.

To complete the proof that $M_k^\loc = M_k$, it remains to show that the local rings of these schemes at the worst point coincide.  Restricting to the affine charts from \s\ref{ss:arzdorf_calcs}, we have
\[
   U^\loc \subset U \to \AA_k^{n-1},
\]
where the second map is $\bigl(\begin{smallmatrix}X_1 \\ X_3\end{smallmatrix}\bigr) \mapsto X_3$ and the composite is the isomorphism \eqref{disp:U^loc_isom}.  The worst point is the point over $0 \in \AA_k^{n-1}$.  It is an easy consequence of the first condition in \eqref{disp:naive_eqns} and the wedge condition \eqref{disp:wedge_eqns} that $U^\wedge$, and a fortiori $U$, is finite over $\AA_k^{n-1}$.  By Nakayama's lemma, it therefore suffices to show that the fiber in $U$ over $0$ is just $\Spec k$.  In other words, we have reduced Theorem \ref{st:main_thm} to the following.

\begin{prop}\label{st:X_1=0}
Let $R$ be a $k$-algebra, and suppose that we have an $R$-point on $U$ such that $X_3 = 0$.  Then $X_1 = 0$.
\end{prop}

\section{Proof of Proposition \ref{st:X_1=0}}\label{s:proof}

In this section we prove Proposition \ref{st:X_1=0}.  We continue with the notation and assumptions of \s\ref{s:special_case}.  
Our strategy is essentially one of computation:  
the main point is to find an explicit $\O_F$-basis for $W(\Lambda_m)_{-1}^{n-1,1}$ (Proposition \ref{st:W(Lambda_m)_coeff_conds}), from which we obtain an explicit $R$-basis for $L_m^{n-1,1}(R)$ whenever $\pi R = 0$ (Corollary \ref{st:im_basis}).

\subsection{Another basis for $V$}  Let
\[
   g_1,\dotsc,g_{2n}
\]
denote the ordered $F$-basis
\begin{multline*}
   e_1 \otimes 1 - \pi e_1 \otimes \pi^{-1},\dotsc, e_n \otimes 1 - \pi e_n \otimes \pi^{-1},\\
   \frac{e_1 \otimes 1 + \pi e_1 \otimes \pi^{-1}} 2, \dotsc, \frac{e_n \otimes 1 + \pi e_n \otimes \pi^{-1}} 2
\end{multline*}
for $V$, which is a split ordered basis for \sform.  Moreover $g_1,\dotsc,g_n$ is a basis for $V_{-\pi}$ and $g_{n+1},\dotsc,g_{2n}$ is a basis for $V_\pi$, which makes $g_1,\dotsc,g_{2n}$ better suited to work with condition \eqref{it:new_cond} than the split basis $f_1,\dotsc,f_{2n}$ in \eqref{disp:PR_basis}.

It is straightforward to see that the change-of-basis matrix expressing the $g_i$'s in terms of the $f_i$'s is contained in $SO_{2n}(F)$, either by explicitly writing out this matrix and computing its determinant, or by noting that the intersection
\[
   \spann\{f_1,\dotsc,f_n\} \cap \spann\{g_1,\dotsc,g_n\} = \spann\{g_{m+1}\}
\]
has even codimension $n-1$ in $\spann\{f_1,\dotsc,f_n\}$ and in $\spann\{g_1,\dotsc,g_n\}$, and then appealing to \cite{paprap09}*{\s7.1.4}.  As discussed in \s\ref{ss:wedge_spin_conds}, this implies that
\[
   W_{\pm 1} = \spann_F\bigl\{\,g_S \pm \sgn(\sigma_S) g_{S^\perp}\bigm| \#S=n\,\bigr\},
\]
where
\[
   g_S \in W
\]
is defined with respect to the basis $g_1,\dotsc,g_{2n}$ as in \eqref{disp:f_S}.

\subsection{Types}

To facilitate working with the subspace $W^{r,s} \subset W$, we make the following definition.

\begin{defn} We say that a subset $S \subset \{1,\dotsc,2n\}$ has \emph{type $(r,s)$} if
\[
   \#(S\cap \{1,\dotsc,n\}) = r
   \quad\text{and}\quad
   \#(S\cap\{n+1,\dotsc,2n\}) = s
\]
\end{defn}
For $r + s = n$, the $g_S$'s for varying $S$ of type $(r,s)$ form a basis for $W^{r,s}$, and it is easy to check that $S$ and $S^\perp$ have the same type.  Hence the following.

\begin{lem}\label{st:W_pm1^r,s_spanning_set}
$W_{\pm 1}^{r,s} = \spann_F\{\, g_S \pm \sgn(\sigma_S) g_{S^\perp} \mid S \text{ is of type } (r,s)\,\}$.\qed
\end{lem}

\begin{rk}\label{rk:sign_sigma_S_type_(n-1,1)}
In proving Proposition \ref{st:X_1=0} we will be interested in $S$ of type $(n-1,1)$.  Such an $S$ is of the form
\[
   S = \bigl\{1,\dotsc, \wh j, \dotsc, n, n+i\bigr\}
\]
for some $i$, $j \leq n$.  By Lemma \ref{st:sign_sigma_S},
\[
   \sgn(\sigma_S) = (-1)^{m+1+\Sigma S} = (-1)^{m+1 + \frac{n(n+1)} 2 -j + n+i} = (-1)^{i + j + 1}.
\]
\end{rk}

\subsection{Weights}

To determine a basis for $W(\Lambda_m)_{-1}^{n-1,1}$, we will need to answer the question of when a linear combination of elements of the form $g_S - \sgn(\sigma_S) g_{S^\perp}$ is contained in $W(\Lambda_m)$.  
The following will help with the bookkeeping.

\begin{defn}
Let $S \subset \{1,\dotsc,2n\}$.  The \emph{weight vector $\mathbf{w}_S$} attached to $S$ is the element of $\ZZ^n$ whose $i$th entry is $\#(S \cap \{i,n+i\})$.
\end{defn}

\begin{rk}\label{rk:(n-1,1)_possible weights}
If $S$ is of type $(n-1,1)$, then there are two possibilities.  The first is that there are $i$ and $j$ such that the $i$th entry of $\mathbf{w}_S$ is $2$, the $j$th entry of $\mathbf{w}_S$ is $0$, and all the other entries of $\mathbf{w}_S$ are $1$.  In this case $S = \{1,\dotsc, \wh j,\dotsc, n, n+i\}$ is uniquely determined by its weight.  The other possibility is that $\mathbf{w}_S = (1,\dotsc,1)$.  In this case all we can say is that $S = \{1,\dotsc, \wh i, \dotsc, n, n+ i\}$ for some $i \in \{1,\dotsc, n\}$.
\end{rk}

\begin{rk}
Of course $\mathbf w_S \in \{0,1,2\}^n$, and $S$ and $S^\perp$ may or may not have the same weight.  For any $S$ of cardinality $n$,
\[
   \mathbf{w}_{S^\perp} + \mathbf{w}_S^\vee = (2,\dotsc,2),
\]
where $\mathbf{w}_S^\vee$ is the vector whose $i$th entry is the $i^\vee$th entry of $\mathbf{w}_S$ for all $i$.
\end{rk}

The reason for introducing the notion of weight is the following obvious fact.

\begin{lem}\label{st:weight_lem}
For $S$ of cardinality $n$, write
\[
   g_S = \sum_{S'}c_{S'}e_{S'},\quad c_{S'} \in F.
\]
Then every $S'$ for which $c_{S'} \neq 0$ has the same weight as $S$.\qed
\end{lem}

By contrast, many different \emph{types} of $S'$ occur in the linear combination in the display.

\subsection{Worst terms}

Let $A$ be a finite-dimensional $F$-vector space, and let $\B$ be an $F$-basis for $A$.

\begin{defn} Let $x = \sum_{b\in \B} c_b b \in A$, with $c_b \in F$.  We say that $c_b b$ is a \emph{worst term for $x$} if
\[
   \ord_\pi (c_b) \leq \ord_\pi (c_{b'}) \quad\text{for all}\quad b' \in \B.
\]
We define
\[
   \WT_{\B}(x) := \sum_{\substack{b \in \B\\ c_b b \text{ is a worst}\\ \text{term for } x}} c_b b.
\]
\end{defn}

Let $\Lambda$ denote the $\O_F$-span of \B in $A$.  Trivially, a nonzero element $x \in A$ is contained in $\Lambda$ if and only if one, hence any, of its worst terms is.  When this is so, and when $R$ is a $k$-algebra, the image of $x$ under the map
\[
   \Lambda \to \Lambda \otimes_{\O_F} R
\]
is the same as the image of $\WT_\B(x)$.

For the rest of the paper we specialize to the case that $A = W$ and \B is the $e_S$-basis for $W$ from \eqref{disp:e_S}.
We abbreviate $\WT_\B$ to $\WT$.  For any $S$ of cardinality $n$, the vector $g_S$ has a unique worst term.  When $S$ has weight $(1,\dotsc,1)$, this is
\begin{multline*}
   (e_{i_1} \otimes 1) \wedge \dotsb \wedge (e_{i_\alpha}\otimes 1) \wedge (-\pi e_{j_1} \otimes \pi^{-1}) \wedge \dotsb \wedge (-\pi e_{j_\beta} \otimes \pi^{-1})\\ \wedge \frac {e_{k_1-n} \otimes 1}2 \wedge \dotsb \wedge \frac {e_{k_\gamma-n} \otimes 1}2 \wedge \frac{\pi e_{l_1-n} \otimes \pi^{-1}}2 \wedge \dotsb \wedge \frac{\pi e_{l_\delta-n} \otimes \pi^{-1}} 2,
\end{multline*}
where we write
\[
   S = \{i_1,\dotsc,i_\alpha, j_1,\dotsc,j_\beta, k_1\dotsc,k_\gamma, l_1,\dotsc,l_\delta\}
\]
with
\begin{multline*}
   i_1 < \dotsb < i_\alpha < m+1 \leq j_1 < \dotsb < j_\beta \leq n\\ 
   < k_1 < \dotsb < k_\gamma < n + m + 1 \leq l_1 < \dotsb < l_\delta.
\end{multline*}
For $S$ of other weight, the worst term of $g_S$ can be made similarly explicit, using the easy fact that
\[
   g_i \wedge g_{n+i} = (e_i \otimes 1) \wedge (\pi e_i \otimes \pi^{-1})
\]
to handle all pairs of the form $i$, $n+i$ that occur in $S$.
Here is
the worst term of $g_S$ in all cases of type $(n-1,1)$.

\begin{lem}\label{st:WT(g_S)} \hfill
\begin{enumerate}
\renewcommand{\theenumi}{\roman{enumi}}
\item
if $S = \{1,\dotsc,\wh i,\dotsc,n,n+i\}$ for some $i < m+1$, then
\[
   \WT(g_S) = \frac {(-1)^{i+m}}2 \pi^{-(m+1)} e_{\{n+1,\dotsc,2n\}}.
\]
\item
if $S = \{1,\dotsc,\wh i,\dotsc,n,n+i\}$ for some $i \geq m+1$, then
\[
   \WT(g_S) = \frac {(-1)^{i+m+1}}2 \pi^{-(m+1)} e_{\{n+1,\dotsc,2n\}}.
\]
\item
if $S = \{1,\dotsc,\wh j,\dotsc,n,n+i\}$ for some $i, j <m+1$ with $i \neq j$, then 
\[
   \WT(g_S) = (-1)^{m+1} \pi^{-m}e_{\{i,n+1,\dotsc, \wh{n+j},\dotsc,2n\}}.
\]
\item
if $S = \{1,\dotsc,\wh j,\dotsc,n,n+i\}$ for some $i < m+1 \leq j$, then 
\[
   \WT(g_S) = (-1)^m \pi^{-(m-1)} e_{\{i,n+1,\dotsc, \wh{n+j},\dotsc,2n\}}.
\]
\item
if $S = \{1,\dotsc,\wh j,\dotsc,n,n+i\}$ for some $j < m+1 \leq i$, then 
\[
   \WT(g_S) = (-1)^{m+1}\pi^{-(m+1)}e_{\{i,n+1,\dotsc, \wh{n+j},\dotsc,2n\}}.
\]
\item
if  $S = \{1,\dotsc,\wh j,\dotsc,n,n+i\}$ for some $i,j \geq m+1$ with $i \neq j$, then
\begin{flalign*}
   \phantom{\qed}& & \WT(g_S) = (-1)^m \pi^{-m} e_{\{i,n+1,\dotsc, \wh{n+j},\dotsc,2n\}}. & & \qed
\end{flalign*}
\end{enumerate}
\end{lem}

\subsection{The lattice $W(\Lambda_m)_{-1}^{n-1,1}$}

In this subsection we determine an $\O_F$-basis for $W(\Lambda_m)_{-1}^{n-1,1}$.  Let $S$ be of type $(n-1,1)$.  Then $S\cap\{n+1,\dotsc,2n\}$ consists of a single element $i_S$.  Define
\[
   S \preccurlyeq S^\perp
   \quad\text{if}\quad
   i_S \leq i_{S^\perp}.
\]
The elements $g_S - \sgn(\sigma_S) g_{S^\perp}$, for varying $S$ of type $(n-1,1)$ and such that $S \preccurlyeq S^\perp$, form a basis for $W_{-1}^{n-1,1}$.  (In particular, note that if $S$ is of type $(n-1,1)$ and $S = S^\perp$, then necessarily $\sgn(\sigma_S)  = -1$, by Remark \ref{rk:sign_sigma_S_type_(n-1,1)}.)  Our task is to determine when a linear combination of 
such elements
is contained in $W(\Lambda_m)$.

As a first step, we calculate the worst terms of $g_S - \sgn(\sigma_S) g_{S^\perp}$.  

\begin{lem}\label{st:WT(g_S_pm_g_S^perp)}
Let $S$ be of type $(n-1,1)$ with $S \preccurlyeq S^{\perp}$.  Then exactly one of the following nine situations holds.
\begin{enumerate}
\item\label{it:case1}
$S = \{1,\dotsc,\wh{m+1},\dotsc,n,n+m+1\}$, $S = S^\perp$, $\mathbf{w}_S = (1,\dotsc,1)$, and
\[
   \WT\bigl(g_S - \sgn(\sigma_S) g_{S^\perp}\bigr) = \WT(2g_S) = \pi^{-(m+1)}e_{\{n+1,\dotsc,2n\}}.
\]
\item\label{it:case2}
$S = \{1,\dotsc,\wh{i^\vee},\dotsc,n,n+i\}$ for some $i < m+1$, $S = S^\perp$, $\mathbf{w}_S \neq (1,\dotsc,1)$, and
\[
   \WT\bigl(g_S - \sgn(\sigma_S) g_{S^\perp}\bigr) = \WT(2g_S) = 2(-1)^{m}\pi^{-(m-1)}e_{\{i,n+1,\dotsc,\wh{i^*},\dotsc,2n\}}.
\]
\item\label{it:case3}
$S = \{1,\dotsc,\wh{i^\vee},\dotsc,n,n+i\}$ for some $i > m+1$, $S = S^\perp$, $\mathbf{w}_S \neq (1,\dotsc,1)$, and
\[
   \WT\bigl(g_S - \sgn(\sigma_S) g_{S^\perp}\bigr) = \WT(2g_S) = 2(-1)^{m+1}\pi^{-(m+1)} e_{\{i,n+1,\dotsc,\wh{i^*},\dotsc,2n\}}.
\]
\item\label{it:case4}
$S = \{1,\dotsc,\wh{i},\dotsc,n,n+i\}$ for some $i < m+1$, $S \neq S^\perp$,
\[
   \mathbf{w}_S = \mathbf{w}_{S^\perp} = (1,\dotsc,1),
\]
and
\[
   \WT\bigl(g_S - \sgn(\sigma_S) g_{S^\perp}\bigr) = (-1)^{m+1}\pi^{-m} \bigl(e_{\{i,n+1,\dotsc,\wh{n+i},\dotsc,2n\}} - e_{\{i^\vee,n+1,\dotsc,\wh{i^*},\dotsc,2n\}} \bigr).
\]
\item\label{it:case5}
$S = \{1,\dotsc,\wh j,\dotsc,n,n+i\}$ for some $i < j^\vee < m+1$; $S \neq S^\perp$; $\mathbf{w}_S$, $\mathbf{w}_{S^\perp}$, and $(1,\dotsc,1)$ are pairwise distinct; and
\begin{multline*}
   \WT\bigl(g_S - \sgn(\sigma_S) g_{S^\perp}\bigr) =\\
      (-1)^m\pi^{-(m-1)}\bigl(e_{\{i,n+1,\dotsc,\wh{n+j},\dotsc,2n\}} + (-1)^{i+j}e_{\{j^\vee,n+1,\dotsc, \wh{i^*},\dotsc,2n\}}\bigr).
\end{multline*}
\item\label{it:case6}
$S = \{1,\dotsc,\wh{m+1},\dotsc,n,n+i\}$ for some $i < m+1$, $S \neq S^\perp$; $\mathbf{w}_S$, $\mathbf{w}_{S^\perp}$, and $(1,\dotsc,1)$ are pairwise distinct; and
\[
   \WT\bigl(g_S - \sgn(\sigma_S) g_{S^\perp}\bigr) = (-1)^{i+1} \pi^{-m} e_{\{m+1,n+1,\dotsc,\wh{i^*},\dotsc,2n\}}.
\]
\item\label{it:case7}
$S = \{1,\dotsc,\wh j,\dotsc,n,n+i\}$ for some $i < m + 1 < j^\vee$; $S \neq S^\perp$; $\mathbf{w}_S$, $\mathbf{w}_{S^\perp}$, and $(1,\dotsc,1)$ are pairwise distinct; and
\begin{multline*}
   \WT\bigl(g_S - \sgn(\sigma_S) g_{S^\perp}\bigr) =\\
      (-1)^{m+1} \pi^{-m} \bigl( e_{\{i,n\dotsc,\wh{n+j},\dotsc,2n \}} + (-1)^{i+j+1}e_{\{j^\vee,n+1,\dotsc, \wh{i^*},\dotsc,2n\}} \bigr).
\end{multline*}
\item\label{it:case8}
$S = \{1,\dotsc,\wh j,\dotsc,n,n+m+1\}$ for some $m + 1 < j^\vee$; $S \neq S^\perp$; $\mathbf{w}_S$, $\mathbf{w}_{S^\perp}$, and $(1,\dotsc,1)$ are pairwise distinct; and
\[
   \WT\bigl(g_S - \sgn(\sigma_S) g_{S^\perp}\bigr) = (-1)^{m+1} \pi^{-(m+1)}e_{\{m+1,n,\dotsc,\wh{n+j},\dotsc,2n\}}.
\]
\item\label{it:case9}
$S = \{1,\dotsc,\wh j,\dotsc,n,n+i\}$ for some $m+1 < i < j^\vee$; $S \neq S^\perp$; $\mathbf{w}_S$, $\mathbf{w}_{S^\perp}$, and $(1,\dotsc,1)$ are pairwise distinct; and
\begin{multline*}
   \WT\bigl(g_S - \sgn(\sigma_S) g_{S^\perp}\bigr) =\\
   (-1)^{m+1}\pi^{-(m+1)}\bigl( e_{\{i,n+1,\dotsc,\wh{n+j},\dotsc,2n\}} + (-1)^{i+j}e_{\{j^\vee,n+1,\dotsc, \wh{i^*},\dotsc,2n\}}\bigr).
\end{multline*}
\end{enumerate}
\end{lem}

\begin{proof}
It is clear that the descriptions of $S$ and $\mathbf{w}_S$ in the nine cases cover all possibilities and are mutually exclusive.  So we have to show that in each case, the calculation of the worst terms is correct.  In cases \eqref{it:case1}--\eqref{it:case3}, this is simply read off from Lemma \ref{st:WT(g_S)}.  The same goes for cases \eqref{it:case5}--\eqref{it:case9}, using also Lemma \ref{st:weight_lem}, and Remark \ref{rk:sign_sigma_S_type_(n-1,1)} to compute the sign of $\sigma_S$.

Thus the only case in which any subtleties arise is \eqref{it:case4}, where $S$ and $S^\perp$ are distinct but have the same weight, and therefore the respective worst terms of $g_S$ and $-\sgn(\sigma_S) g_{S^\perp}$ may, and in fact do, cancel.  Here is the basic calculation, which we formulate as a separate lemma.

\begin{lem}\label{st:g_S_nontriv_wt_1}
Let $S = \{1,\dotsc,\wh{i},\dotsc, n, n+i\}$ for some $i < m+1$.  Then
\begin{multline*}
   g_S - \sgn(\sigma_S)  g_{S^\perp} = (-1)^i g_1 \wedge \dotsb \wedge \wh{g_i} \wedge \dotsb \wedge \wh{g_{i^\vee}} \wedge \dotsb \wedge g_n\\
    \wedge \bigl[(e_i \otimes 1) \wedge (e_{i^\vee} \otimes 1) - (\pi^{-1}e_i \otimes 1) \wedge (\pi e_{i^\vee} \otimes 1)\bigr].
\end{multline*}
\end{lem}

\begin{proof}
We have
\[
   S^\perp = \{1,\dotsc,\wh{i^\vee},\dotsc,n,i^*\}
\]
and, by Remark \ref{rk:sign_sigma_S_type_(n-1,1)},
\[
   -\sgn(\sigma_S)  = (-1)^{2i} = 1.
\]
Hence
\begin{align*}
   g_S - \sgn(\sigma_S)  g_{S^\perp} 
      &= g_S + g_{S^\perp}\\
      &= g_1 \wedge \dots \wedge \wh{g_{i}} \wedge \dotsb \wedge g_n \wedge g_{n+i} + g_1 \wedge \dotsb \wedge \wh{g_{i^\vee}} \wedge \dotsb \wedge g_n \wedge g_{i^*}\\
      &= (-1)^{n-i^\vee} g_1 \wedge \dotsb \wedge \wh{g_i} \wedge \dotsb \wedge \wh{g_{i^\vee}} \wedge \dotsb \wedge g_n \wedge g_{i^\vee} \wedge g_{n+i}\\
      &\quad\quad + (-1)^{n-1-i} g_1 \wedge \dotsb \wedge \wh{g_i} \wedge \dotsb \wedge \wh{g_{i^\vee}} \wedge \dotsb \wedge g_n \wedge g_{i} \wedge g_{i^*}\\
      &= (-1)^ig_1 \wedge \dotsb \wedge \wh{g_i} \wedge \dotsb \wedge \wh{g_{i^\vee}} \wedge \dotsb \wedge g_n\\
      &\qquad\qquad\qquad\qquad\qquad\qquad\qquad\qquad \wedge [g_i \wedge g_{i^*} - g_{i^\vee} \wedge g_{n+i}].
\end{align*}
It is elementary to verify that
\[
   g_i \wedge g_{i^*} - g_{i^\vee} \wedge g_{n+i} = (e_i \otimes 1) \wedge(e_{i^\vee} \otimes 1) - (\pi^{-1}e_i \otimes 1) \wedge (\pi e_{i^\vee} \otimes 1),
\]
which completes the proof.
\end{proof}

Returning to the calculation of the worst terms in case \eqref{it:case4} in Lemma \ref{st:WT(g_S_pm_g_S^perp)}, Lemma \ref{st:g_S_nontriv_wt_1} implies that
\begin{align*}
   \WT\bigl(g_S - {}&\sgn(\sigma_S) g_{S^\perp}\bigr)\\
    ={} &(-1)^i (e_1 \otimes 1) \wedge \dotsb \wedge \wh{(e_i \otimes 1)} \wedge \dotsb \wedge (e_m \otimes 1)\\
   &\qquad\wedge (-\pi e_{m+1} \otimes \pi^{-1}) \wedge \dotsb \wedge \wh{(-\pi e_{i^\vee} \otimes \pi^{-1})} \wedge \dotsb \wedge (-\pi e_n \otimes \pi^{-1})\\
   &\qquad\qquad\wedge \bigl[(e_i \otimes 1) \wedge (e_{i^\vee} \otimes 1) - (\pi^{-1}e_i \otimes 1) \wedge (\pi e_{i^\vee} \otimes 1)\bigr]\\
   ={} & (-1)^{m+1}\pi^{-m}\bigl( e_{\{i,n+1,\dotsc,\wh{n+i},\dotsc,2n\}} - e_{\{i^\vee, n+1,\dotsc,\wh{i^*},\dotsc,2n\}} \bigr).
\end{align*}
This completes the proof of Lemma \ref{st:WT(g_S_pm_g_S^perp)}.
\end{proof}

The following asserts that appropriate multiples of the elements $g_S - \sgn(\sigma_S) g_{S^\perp}$, for $S \preccurlyeq S^\perp$ of type $(n-1,1)$, form an $\O_F$-basis of $W(\Lambda_m)_{-1}^{n-1,1}$.

\begin{prop}\label{st:W(Lambda_m)_coeff_conds}
Let $w \in W_{-1}^{n-1,1}$, and write
\[
   w = \sum_{\substack{S \preccurlyeq S^\perp\\ \text{of type } (n-1,1)}}a_S\bigl(g_S - \sgn(\sigma_S) g_{S^\perp}\bigr),
   \quad a_S \in F.
\]
Then $w \in W(\Lambda_m)_{-1}^{n-1,1}$ $\iff$
\begin{enumerate}
\item\label{it:cond_1}
if $S = \{1,\dotsc,\wh{m+1},\dotsc,n,n+m+1\}$, then $\ord_\pi(a_S) \geq m+1$;
\item\label{it:cond_2}
if $S = \{1,\dotsc,\wh{i^\vee},\dotsc,n,n+i\}$ for some $i < m+1$, then $\ord_\pi(a_S) \geq m-1$;
\item\label{it:cond_3}
if $S = \{1,\dotsc,\wh{i^\vee},\dotsc,n,n+i\}$ for some $i > m+1$, then $\ord_\pi(a_S) \geq m+1$;
\item\label{it:cond_4}
if $S = \{1,\dotsc,\wh{i},\dotsc,n,n+i\}$ for some $i < m+1$, then $\ord_\pi(a_S) \geq m$;
\item\label{it:cond_5}
if $S = \{1,\dotsc,\wh j,\dotsc,n,n+i\}$ for some $i < j^\vee < m+1$, then $\ord_\pi(a_S) \geq m-1$;
\item\label{it:cond_6}
if $S = \{1,\dotsc,\wh j,\dotsc,n,n+i\}$ for some $i < m+1 \leq j^\vee$, then $\ord_\pi(a_S) \geq m$; and
\item\label{it:cond_7}
if  $S = \{1,\dotsc,\wh j,\dotsc,n,n+i\}$ for some $m+1 \leq i < j$, then $\ord_\pi(a_S) \geq m+1$.
\end{enumerate}
\end{prop}

\begin{proof}
The implication $\Longleftarrow$ is immediate from Lemma \ref{st:WT(g_S_pm_g_S^perp)}.  For the implication $\Longrightarrow$, assume $w \in W(\Lambda_m)_{-1}^{n-1,1}$, and write
\[
   w = \sum_{\mathbf{w}\in\ZZ^n} 
       \sum_{\substack{S \preccurlyeq S^\perp\\ \text{of type } (n-1,1)\\ \text{and weight }\mathbf{w}}}a_S\bigl(g_S - \sgn(\sigma_S) g_{S^\perp}\bigr).
\]
By Lemma \ref{st:weight_lem}, $w \in W(\Lambda_m)_{-1}^{n-1,1}$ $\iff$ for each $\mathbf{w}$,
\begin{equation}\label{disp:fixed_w_lin_comb}
   \sum_{\substack{S \preccurlyeq S^\perp\\ \text{of type } (n-1,1)\\ \text{and weight }\mathbf{w}}}a_S\bigl(g_S - \sgn(\sigma_S) g_{S^\perp}\bigr) \in W(\Lambda_m)_{-1}^{n-1,1}.
\end{equation}
Thus we reduce to working weight by weight.  If $\mathbf{w} \neq (1,\dotsc,1)$, then by Remark \ref{rk:(n-1,1)_possible weights} there is at most one $S$ of type $(n-1,1)$ and weight $\mathbf w$.  This and Lemma \ref{st:WT(g_S_pm_g_S^perp)} imply \eqref{it:cond_2}, \eqref{it:cond_3}, \eqref{it:cond_5}, \eqref{it:cond_6}, and \eqref{it:cond_7}.

The case $\mathbf{w} = (1,\dotsc,1)$ requires a finer analysis, since in this case all the sets
\[
   S_i := \bigl\{1,\dotsc,\wh{i},\dotsc,n,n+i\bigr\} \quad\text{for}\quad i = 1,\dotsc,m+1
\]
occur as summation indices in \eqref{disp:fixed_w_lin_comb}.  First consider the set $S_{m+1}$.  For $i < m+1$, write $g_{S_i} - \sgn(\sigma_{S_i}) g_{S_i^\perp}$ as a linear combination of $e_{S}$'s.  By Lemma \ref{st:g_S_nontriv_wt_1}, every $e_{S}$ that occurs 
in this linear combination must involve wedge factors of either $e_i \otimes 1$ and $e_{i^\vee} \otimes 1$ together, or $\pi^{-1} e_i \otimes 1$ and $\pi e_{i^\vee} \otimes 1$ together.  By Lemma \ref{st:WT(g_S_pm_g_S^perp)},
\[
   \WT\bigl(g_{S_{m+1}} - \sgn(\sigma_{S_{m+1}}) g_{S_{m+1}}\bigr) = \pi^{-(m+1)}e_{\{n+1,\dotsc,2n\}}
\]
involves no such $e_S$.  
This implies \eqref{it:cond_1}.

By the same argument, for fixed $i < m+1$, the $e_S$-basis vectors that occur in $\WT\bigl(g_{S_i} - \sgn(\sigma_{S_i}) g_{S_i^\perp}\bigr)$, namely $e_{\{i,n+1,\dotsc,\wh{n+i},\dotsc,2n\}}$ and $e_{\{i^\vee,n+1,\dotsc,\wh{i^*},\dotsc,2n\}}$, don't occur in $g_{S_j} - \sgn(\sigma_{S_j}) g_{S_j^\perp}$ for $j \neq i$, $m+1$.  It follows easily from this, the fact that we've already shown that $\ord_\pi(a_{S_{m+1}}) \geq m+1$, and Lemma \ref{st:WT(g_S_pm_g_S^perp)} that $\ord_\pi(a_{S_i}) \geq m$.
\end{proof}

The following consequence is what we'll need for the proof of Proposition \ref{st:X_1=0}.

\begin{cor}\label{st:im_basis}
Let $R$ be a $k$-algebra.  Then $L_m^{n-1,1}(R) \subset W(\Lambda_m) \otimes_{\O_F} R$
is a free $R$-module on the basis elements
\begin{enumerate}
\item\label{it:im1} $e_{\{n+1,\dotsc,2n\}}$;
\item\label{it:im2} $e_{\{i,n+1,\dotsc,\wh{i^*},\dotsc,2n\}}$ for $i = 1,\dotsc, \wh{m+1},\dotsc,n$;
\item\label{it:im3} $e_{\{i,n+1,\dotsc,\wh{n+i},\dotsc,2n\}} - e_{\{i^\vee,n+1,\dotsc,\wh{i^*},\dotsc,2n\}}$ for $i < m+1$;
\item\label{it:im4} $e_{\{i,n+1,\dotsc,\wh{n+j},\dotsc,2n\}} + (-1)^{i+j}e_{\{j^\vee,n+1,\dotsc, \wh{i^*},\dotsc,2n\}}$ for $i < j^\vee < m+1$ and $m+1 < i < j^\vee \leq n$;
\item\label{it:im5} $e_{\{m+1,n+1,\dotsc,\wh{n+i},\dotsc,2n\}}$ for $i = 1,\dotsc, \wh{m+1},\dotsc,n$; and
\item\label{it:im6} $e_{\{i,n+1,\dotsc,\wh{n+j},\dotsc,2n\}} + (-1)^{i+j+1}e_{\{j^\vee,n+1,\dotsc, \wh{i^*},\dotsc,2n\}}$ for $i < m+1 < j^\vee \leq n$.
\end{enumerate}
\end{cor}

\begin{proof}
Immediate from Lemma \ref{st:WT(g_S_pm_g_S^perp)} and Proposition \ref{st:W(Lambda_m)_coeff_conds}, using that under the canonical map $W(\Lambda_m)_{-1}^{n-1,1} \to W(\Lambda_m) \otimes_{\O_F} R$, the image of any element $w$ is the same as the image of $\WT(w)$.
\end{proof}

\subsection{Proof of Proposition \ref{st:X_1=0}}

We now prove Proposition \ref{st:X_1=0}.  Let $R$ be a $k$-algebra, and let $X_1$ be an $(n-1) \times (n-1)$-matrix with entries in $R$ satisfying
\[
   X_1 = -JX_1^\Rt J, \quad X_1^2 = 0, \quad \sideset{}{^2}\bigwedge X_1 = 0, \quad\text{and}\quad \tr X_1 = 0,
\]
and such that when $X_1$ is plugged into \eqref{disp:U^Gr} along with $X_2$, $X_3$, and $X_4$ all $0$, the resulting $R$-point on $U^\wedge$ lies in $U$, i.e.\ it satisfies \eqref{it:new_cond}.  Our problem is to show that $X_1 = 0$.

In fact we will show that just the conditions $X_1 = -JX_1^\Rt J$ and \eqref{it:new_cond} imply that $X_1 = 0$.  Decompose $X_1$ into $m \times m$ blocks
\[
   X_1 = \begin{pmatrix}
          A & B\\
          C & D\\
       \end{pmatrix}.
\]
Then 
$X_1 = -JX_1^tJ$ is equivalent to
\begin{equation}\label{disp:LM_conds}
   D = A^\ad, \quad
   B = -B^\ad, \quad\text{and}\quad
   C = -C^\ad,
\end{equation}
where the superscript $\ad$ means to take the transpose across the antidiagonal, or in other words, to take the adjoint with respect to the standard split symmetric form.  We are going to show that \eqref{it:new_cond} imposes the \emph{same conditions as in \eqref{disp:LM_conds} except with opposite signs.}  Since $\charac k \neq 2$, this will imply that $A = B = C = D = 0$.

Let $v \in W(\Lambda_m) \otimes_{\O_F} R$ denote the wedge product (say from left to right) of the $n$ columns of the matrix \eqref{disp:U^Gr}, where each column is regarded as an $R$-linear combination of Arzdorf's basis elements \eqref{disp:Arzdorf_basis}.  Condition \eqref{it:new_cond} is that
\[
   v \in L_m^{n-1,1}(R).
\]

We begin by analyzing the implications of this condition on the entries of $A$ and $D$.  Let
\[
   1 \leq i,j \leq m,
\]
and let $a_{ij}$ and $d_{ij}$ denote the $(i,j)$-entries of $A$ and $D$, respectively.  Let
\[
   S := \{m+1+i,n+1,\dotsc,(n+m+1+j)\sphat\;,\dotsc,2n\}.
\]
Then
\[
   S^\perp = \{m+1-j,n+1,\dotsc,(n+m+1-i)\sphat\;,\dotsc,2n\}.
\]
Writing $v$ as a linear combination of $e_{S'}$'s, the $e_S$-term in $v$ is
\[
\begin{split}
   & (-1)^{1+j}a_{ij}(e_{m+1+i} \otimes 1) \\ 
   &\qquad\wedge (\pi e_{m+2} \otimes 1) \wedge \dotsb \wedge (\pi e_{m+1+j} \otimes 1)\sphat\, \wedge \dotsb \wedge (\pi e_n \otimes 1)\\
   &\qquad\qquad\wedge (e_1 \otimes 1) \wedge \dotsb \wedge (e_m \otimes 1) \wedge (\pi e_{m+1} \otimes 1)\\
   &\qquad\qquad\qquad\qquad\qquad\qquad\qquad\qquad\qquad\qquad = (-1)^{1+j+(m+1)(m-1)} a_{ij}e_S\\
   &\qquad\qquad\qquad\qquad\qquad\qquad\qquad\qquad\qquad\qquad = (-1)^{m+j}a_{ij}e_S,
\end{split}
\]
and the $e_{S^\perp}$-term in $v$ is
\[
\begin{split}
   & (-1)^{1+n-i}d_{m+1-j,m+1-i}(\pi^{-1}e_{m+1-j} \otimes 1)\\
   &\qquad \wedge (\pi e_{m+2} \otimes 1) \wedge \dotsb \wedge (\pi e_n \otimes 1)\\
   &\qquad\qquad\wedge (e_1 \otimes 1) \wedge \dotsb \wedge (e_{m+1-i} \otimes 1)\sphat\, \wedge \dotsb \wedge (e_m \otimes 1) \wedge (\pi e_{m+1} \otimes 1)\\
   &\qquad\qquad\qquad\qquad\qquad\qquad\qquad\qquad\qquad = (-1)^{1+n-i+m^2} d_{m+1-j,m+1-i}e_{S^\perp}\\
   &\qquad\qquad\qquad\qquad\qquad\qquad\qquad\qquad\qquad =    (-1)^{m+i}d_{m+1-j,m+1-i}e_{S^\perp}.
\end{split}
\]
If $i = j$, then using Corollary \ref{st:im_basis}, especially part \eqref{it:im3} applied with $m+1-i$ in place of $i$, the condition $v \in L_m^{n-1,1}(R)$ requires that
\[
   d_{m+1-i,m+1-i} = -a_{ii}.
\]
If $i \neq j$, then using Corollary \ref{st:im_basis}, especially part \eqref{it:im6} applied with $m+1-j$ in place of $i$ and $m+1-i$ in place of $j$, we similarly find that
\[
   d_{m+1-j,m+1-i} = -a_{ij}.
\]
Hence $D^\ad = -A$, as desired.

We next turn to $B$.  By \eqref{disp:LM_conds} the antidiagonal entries of $B$ are $0$.  Let
\[
   1 \leq i < m+1-j \leq m,
\]
and let $b_{ij}$ denote the $(i,j)$-entry of $B$.  Let
\[
   S := \{m+1+i,n+1,\dotsc,\wh{n+j},\dotsc,2n\}.
\]
Then
\[
   S^\perp = \{j^\vee,n+1,\dotsc,(n+m+1-i)\sphat\; ,\dotsc,2n\},
\]
which differs from $S$ since $i < m+1-j$.  The $e_S$-term in $v$ is
\[
\begin{split}
   & (-1)^{1+m+j}b_{ij}(e_{m+1+i} \otimes 1) \wedge (\pi e_{m+2} \otimes 1) \wedge \dotsb \wedge (\pi e_n \otimes 1)\\ 
   &\qquad\wedge (e_1 \otimes 1) \wedge \dotsb \wedge\wh{(e_{j} \otimes 1)} \wedge \dotsb \wedge (e_m \otimes 1) \wedge (\pi e_{m+1} \otimes 1)\\
   &\qquad\qquad\qquad\qquad\qquad\qquad\qquad\qquad\qquad\qquad\qquad = (-1)^{1+m+j+m^2} b_{ij}e_S\\
   &\qquad\qquad\qquad\qquad\qquad\qquad\qquad\qquad\qquad\qquad\qquad = (-1)^{1+j}b_{ij}e_S,
\end{split}
\]
and the $e_{S^\perp}$-term in $v$ is
\[
\begin{split}
   & (-1)^{1+m+(m+1-i)}b_{m+1-j,m+1-i}(e_{j^\vee} \otimes 1) \wedge (\pi e_{m+2} \otimes 1) \wedge \dotsb \wedge (\pi e_n \otimes 1)\\ 
   &\qquad\wedge (e_1 \otimes 1) \wedge \dotsb \wedge (e_{m+1-i} \otimes 1)\sphat\, \wedge \dotsb \wedge (e_m \otimes 1) \wedge (\pi e_{m+1} \otimes 1)\\
   &\qquad\qquad\qquad\qquad\qquad\qquad\qquad\qquad\qquad\qquad = (-1)^{i+m^2} b_{m+1-j,m+1-i}e_{S^\perp}\\
   &\qquad\qquad\qquad\qquad\qquad\qquad\qquad\qquad\qquad\qquad = (-1)^{i+m}b_{m+1-j,m+1-i}e_{S^\perp}.
\end{split}
\]
As above, this time using Corollary \ref{st:im_basis}\eqref{it:im4} with $m+1+i$ in place of $i$, we find that $v \in L_m^{n-1,1}(R)$ requires that
\[
   b_{m+1-j,m+1-i} = b_{ij}.
\]
Hence $B^\ad = B$, as desired.

We finally turn to $C$, for which the argument is almost identical to the one for $B$.  By \eqref{disp:LM_conds} the antidiagonal entries of $C$ are $0$.  Let
\[
   1 \leq i < m+1-j \leq m,
\]
and let $c_{ij}$ denote the $(i,j)$-entry of $C$.  Let
\[
   S := \{i,n+1,\dotsc,(n+m+1+j)\sphat\; ,\dotsc,2n\}.
\]
Then
\[
   S^\perp = \{m+1-j,n+1,\dotsc,\wh{i^*},\dotsc,2n\} \neq S.
\]
  The $e_S$-term in $v$ is
\[
\begin{split}
   & (-1)^{1+j}c_{ij}(\pi^{-1}e_{i} \otimes 1)\\
   &\qquad \wedge (\pi e_{m+2} \otimes 1) \wedge \dotsb \wedge (\pi e_{m+1+j} \otimes 1)\sphat\, \wedge \dotsb \wedge (\pi e_n \otimes 1)\\ 
   &\qquad\qquad\wedge (e_1 \otimes 1) \wedge \dotsb \wedge (e_m \otimes 1) \wedge (\pi e_{m+1} \otimes 1)\\
   &\qquad\qquad\qquad\qquad\qquad\qquad\qquad\qquad\qquad\qquad = (-1)^{1+j+(m+1)(m-1)} c_{ij}e_S\\
   &\qquad\qquad\qquad\qquad\qquad\qquad\qquad\qquad\qquad\qquad = (-1)^{j+m}c_{ij}e_S,
\end{split}
\]
and the $e_{S^\perp}$-term in $v$ is
\[
\begin{split}
   & (-1)^{1+m+1-i}c_{m+1-j,m+1-i}(\pi^{-1}e_{m+1-j} \otimes 1)\\
   &\qquad \wedge (\pi e_{m+2} \otimes 1) \wedge \dotsb \wedge (\pi e_{i^\vee} \otimes 1)\sphat\, \wedge \dotsb \wedge (\pi e_n \otimes 1)\\ 
   &\qquad\qquad\wedge (e_1 \otimes 1) \wedge \dotsb \wedge (e_m \otimes 1) \wedge (\pi e_{m+1} \otimes 1)\\
   &\qquad\qquad\qquad\qquad\qquad\qquad\qquad\qquad\qquad = (-1)^{m+i+(m+1)(m-1)} c_{m+1-j,m+1-i}e_{S^\perp}\\
   &\qquad\qquad\qquad\qquad\qquad\qquad\qquad\qquad\qquad = (-1)^{i+1}c_{m+1-j,m+1-i}e_{S^\perp}.
\end{split}
\]
Using Corollary \ref{st:im_basis}\eqref{it:im4} with $m+1+j$ in place of $j$, we find as before that $v \in L_m^{n-1,1}(R)$ requires that
\[
   c_{m+1-j,m+1-i} = c_{ij}.
\]
Hence $C^\ad = C$, as desired.  This completes the proof of Proposition \ref{st:X_1=0}, which in turn completes the proof of Theorem \ref{st:main_thm}.\qed

\section{Further remarks}\label{s:remarks}

In this final section of the paper we collect a few general remarks.  We return to the setting of arbitrary $n$, signature $(r,s)$, and $I$ satisfying \eqref{disp:I_cond}.

\subsection{Relation to the Kottwitz condition}\label{ss:kottwitz}

It is notable that the Kottwitz condition didn't intervene explicitly in the proof of Proposition \ref{st:X_1=0} or, more generally, of Theorem \ref{st:main_thm}.  There is a good reason for this, as we shall now see.

For $R$ an $\O_F$-algebra, define the following condition on an $R$-point $(\F_i)_{i}$ of $M_I^\naive$:
\begin{enumerate}
\renewcommand{\theenumi}{K$_n$}
\item\label{it:K_n}
for all $i$, the line $\bigwedge_R^n \F_i \subset W(\Lambda_i) \otimes_{\O_F} R$ is contained in
\begin{equation}\label{disp:K_n_im}
   \im \bigl[ \bigl(W^{r,s} \cap W(\Lambda_i)\bigr) \otimes_{\O_F} R \to W(\Lambda_i) \otimes_{\O_{F}} R \bigr].
\end{equation}
\end{enumerate}
As usual, this defines a condition on $M_I^\naive$ when $r \neq s$, and when $r = s$ it descends from $M_{I,\O_F}^\naive$ to $M_I^\naive$, since in this case $W^{r,s}$ is Galois-stable.

Condition \eqref{it:K_n} is trivially implied by \eqref{it:new_cond}.  In the generic fiber, \eqref{it:K_n} is equivalent to the Kottwitz condition, and in general we have the following.

\begin{lem}\label{st:K_n==>Kottwitz}
Condition \eqref{it:K_n}, and a fortiori condition \eqref{it:new_cond}, implies the Kott\-witz condition.
\end{lem}

\begin{proof}
Let $T$ be a formal variable.  For any $w \in W^{r,s}$, the identity
\[
   \Bigl[\sideset{}{^n}\bigwedge (T - \pi \otimes 1)\Bigr] \cdot w = (T+\pi)^r(T-\pi)^s w
\]
holds true.  Hence this identity holds true for any $w$ in the image \eqref{disp:K_n_im}.
\end{proof}

\subsection{Wedge power analogs}\label{ss:wedge_power_analogs}

More generally, one can formulate an analog of condition \eqref{it:K_n} for any wedge power.  For $1 \leq l \leq n$, let
\[
   \tensor[^l]W{^{r,s}} := \bigoplus_{\substack{j+k=l\\ j\leq r\\ k \leq s}}
      \Bigl(\sideset{}{_F^j}\bigwedge V_{-\pi} \otimes_F \sideset{}{_F^k}\bigwedge V_\pi\Bigr) \subset \sideset{}{_F^l}\bigwedge V.
\]
In terms of our previous notation, $\tensor[^n]W{^{r,s}} = W^{r,s}$.  For any $\O_{F_0}$-lattice $\Lambda$ in $F^n$, let
\[
   \tensor[^l]W{} (\Lambda)^{r,s} := \tensor[^l]W{^{r,s}} \cap \sideset{}{_{\O_F}^l}\bigwedge (\Lambda \otimes_{\O_{F_0}} \O_F)
\]
(intersection in $\bigwedge_F^lV$).  Then, for $R$ an $\O_F$-algebra and $(\F_i)_i$ an $R$-point on the naive local model, we formulate the condition that
\begin{enumerate}
\renewcommand{\theenumi}{K$_l$}
\item\label{it:K_l}
for all $i$, the subbundle $\bigwedge_R^l \F_i \subset \bigwedge_R^l (\Lambda \otimes_{\O_{F_0}} R)$ is contained in
\[
   \im \Bigl[ \tensor[^l]W{} (\Lambda_i)^{r,s} \otimes_{\O_F} R \to \sideset{}{_R^l}\bigwedge (\Lambda \otimes_{\O_{F_0}} R) \Bigr].
\]
\end{enumerate}

As usual, this condition descends from $M_{I,\O_F}^\naive$ to $M_I^\naive$ when $r = s$.  Conditions \eqref{it:K_l} for $l = 1,\dotsc$, $n$ are closed conditions on $M_I^\naive$, and they are all implied by the Kottwitz condition in the generic fiber.  Therefore they all hold on $M_I^\loc$.  

We do not know the relation between the \eqref{it:K_l} conditions in general.  If our conjecture that $M_I = M_I^\loc$ holds true, then they are all implied by conditions \eqref{it:LM1}--\eqref{it:new_cond}.  Here is a statement pointing in the other direction.

\begin{lem}
Assume $r \neq s$.  Then conditions $(\mathrm{K}_{r+1})$ and $(\mathrm{K}_{s+1})$ imply the wedge condition \eqref{it:wedge_cond}.
\end{lem}

\begin{proof}
As in the proof of Lemma \ref{st:K_n==>Kottwitz}, this is an immediate consequence of the fact that for any $w \in \tensor[^{s+1}]W{^{r,s}}$ and $w' \in \tensor[^{r+1}]W{^{r,s}}$, we have
\[
   \Bigl[ \sideset{}{^{s+1}}\bigwedge (\pi \otimes 1 + 1 \otimes \pi)\Bigr] \cdot w = 0
    \quad\text{and}\quad
	\Bigl[ \sideset{}{^{r+1}}\bigwedge (\pi \otimes 1 - 1 \otimes \pi)\Bigr] \cdot w' = 0.\qedhere
\]
\end{proof}

\begin{rk}\label{rk:newcond==>wedge}
It is perhaps interesting to note that the flatness of $M_{\{m\}}$ in the setting of Theorem \ref{st:main_thm} follows from only conditions \eqref{it:LM1}--\eqref{it:LM4} in the definition of the naive local model and our new condition \eqref{it:new_cond}.  Indeed, the only other condition needed to make our proof go through is the wedge condition \eqref{it:wedge_cond} in the special fiber, which is used in the reduction argument given in \s\ref{ss:reduction} and in the proof of topological flatness.  
Since the elements $e_S$ occurring in the statement of Corollary \ref{st:im_basis} are all for $S$ of type $(1,n-1)$ and $(0,n)$, condition \eqref{it:new_cond} implies that the $2 \times 2$-minors of the matrix $X$ in \eqref{disp:U^Gr} must vanish, which shows that \eqref{it:new_cond} implies the wedge condition inside the open subscheme $U^\naive$ of the special fiber defined in \eqref{U's}.  Therefore \eqref{it:new_cond} implies the wedge condition on the entire special fiber, since this neighborhood contains the worst point and \eqref{it:new_cond} is invariant under polarized lattice chain automorphisms in the sense of \cite{rapzink96}*{\s3}.  One can similarly show that \eqref{it:new_cond} implies the wedge condition in the special fiber for any signature (still with $n$ odd and $I = \{m\}$) by proving the suitable analog of Corollary \ref{st:im_basis}.
\end{rk}

\subsection{Application to Shimura varieties}\label{ss:Sh}

In this subsection we give the most immediate application of Theorem \ref{st:main_thm} to Shimura varieties. Indeed, we are simply going to make explicit 
Rapoport and Zink's general definition of integral models of PEL Shimura varieties \cite[Def.~6.9]{rapzink96},
in the particular case of a unitary similitude group ramified and quasi-split at $p$ and of signature $(n-1,1)$ for odd $n$, where the level structure at $p$ is the special maximal parahoric subgroup corresponding to $I = \{m\}$ (in the sense of \cite[\s1.2.3(a)]{paprap09}); and in a way that furthermore incorporates our condition \eqref{it:new_cond}.  This is also the setting of \cite[\s1]{paprap09} (which allows any $n$, $I$, and signature), and differs only slightly from \cite[\s3]{pappas00} (which allows any $n$ and signature and takes $I = \{0\}$).  By the general formalism of local models, Richarz's smoothness result \cite[Prop.~4.16]{arzdorf09} and Theorem \ref{st:main_thm} will show that the moduli space we write down is smooth.

Let $K/\QQ$ be an imaginary quadratic field in which the prime $p \neq 2$ ramifies, and take $F_0 = \QQ_p$ and $F = K \otimes_\QQ \QQ_p$.  Let $a \mapsto \ol a$ denote the nontrivial element of $\Gal(K/\QQ)$.  Let $n$ be odd, and let $H$ be an $n$-dimensional $K/\QQ$-Hermitian space of signature $(n-1,1)$ (as in \cite[\s1.1]{paprap09}) such that the $F/\QQ_p$-Hermitian space $H \otimes_\QQ \QQ_p$ is split, i.e.~it has a basis $e_1,\dotsc,e_n$ such that \eqref{split} holds.  Fix such a basis, and define the lattice chain
\[
   \Lambda_{\{m\}} := \{\Lambda_i\}_{i \in \pm m + n\ZZ}
\]
in $H \otimes \QQ_p$ with respect to it, with $\Lambda_i$ as in \eqref{Lambda_i}. Let $G := GU(H)$.  

We now use (essentially) the notation and terminology of \cite[\s6.1--6.9]{rapzink96} for abelian schemes.  For $R$ a ring, let $AV(R)$ denote the category of abelian schemes with $\O_K$-action over $\Spec R$ up to prime-to-$p$ isogeny.  Thus an object in $AV(R)$ consists of a pair $(A,\iota)$, where $A$ is an abelian scheme over $\Spec R$ and $\iota$ is a ring homomorphism $\O_K \to \End_R(A) \otimes_\ZZ \ZZ_{(p)}$; and the morphisms $(A_1,\iota_1) \to (A_2,\iota_2)$ in $AV(R)$ consist of the elements in $\Hom_R(A_1,A_2)\otimes \ZZ_{(p)}$ commuting with the $\O_K$-actions.  The dual of an object $(A,\iota)$ in $AV(R)$ is the pair $(A^\vee,\iota^\vee)$, where $A^\vee$ is the dual abelian scheme and $\iota^\vee(a) := \iota(\ol a)^\vee$.  A \QQ-homogeneous polarization is a collection $\ol \lambda$ of quasi-isogenies $A \to A^\vee$ such that any two elements in $\ol\lambda$ differ Zariski-locally on $\Spec R$ by an element in $\QQ^\times$, and such that $\ol\lambda$ contains a polarization in $AV(R)$.  Note that, by definition of the $\O_K$-action on the dual abelian scheme, the Rosati involution attached to any polarization induces the nontrivial Galois automorphism on $K$.

Now we define the moduli space.  Let $C^p \subset G(\AA_f^p)$ be a sufficiently small open compact subgroup.  We denote by $\S_{C^p}$ the moduli problem over $\Spec \O_F$ which associates to each $\O_F$-algebra $R$ the set of isomorphism classes of quadruples $(A,\iota,\ol\lambda,\ol\eta)$ consisting of
\begin{enumerate}
\item
an object $(A,\iota)$ in $AV(R)$ such that $A$ has relative dimension $n$ over $\Spec R$;
\item\label{plzn}
a $\QQ$-homogeneous polarization $\ol\lambda$ on $(A,\iota)$ containing a polarization $\lambda$ such that $\ker\lambda \subset A[\iota(\pi)]$ of height $n-1$, where $\pi$ is a uniformizer in $\O_K \otimes_\ZZ \ZZ_{(p)}$; and
\item
a $C^p$-level structure $\ol\eta\colon H_1(A,\AA_f^p) \isoarrow H \otimes_\QQ \AA_f^p \bmod C^p$ that respects the bilinear forms on both sides up to a constant in $(\AA_f^p)^\times$ (see \citelist{\cite{rapzink96}*{\s6.8--6.9}\cite{kott92}*{\s5}} for details);
\end{enumerate}
subject to the following condition, which is the translation to this setting of \eqref{it:new_cond}.  Given such a quadruple $(A,\iota,\ol\lambda,\ol\eta)$, let $\lambda$ and $\pi$ be as in \eqref{plzn}.  Since $A$ is polarized and of relative dimension $n$, its $p$-divisible group has height $2n$.  Hence $\iota(\pi)$ is an isogeny (in $AV(R)$) of height $n$, since $p$ ramifies in $K$.  Hence there exists a unique isogeny $\rho\colon A^\vee \to A$ (which has height $1$) such that the composite
\[
   A \xra\lambda A^\vee \xra\rho A
\]
is $\iota(\pi)$.  Extending this diagram periodically, we obtain an $\L$-set of abelian varieties, in the terminology of \cite[Def.~6.5]{rapzink96}, for $\L$ the $\O_F$-lattice chain $\Lambda_{\{m\}}$.  Writing $f$ for the structure morphism $f\colon A \to \Spec R$, define $M(A) := (R^1f_*(\Omega^\bullet_{A/R}))^\vee$ to be the $R$-linear dual of the first de Rham cohomology module of $A$.  Likewise define $M(A^\vee)$.  Then $M(A)$ and $M(A^\vee)$ are finite locally free $R$-modules of rank $2n$, and the isogenies $\lambda$ and $\rho$ induce a diagram of $\O_F \otimes_{\ZZ_p} R$-modules
\[
   M(A) \xra{\lambda_*} M(A^\vee) \xra{\rho_*} M(A),
\]
which extends periodically to a chain of $\O_F \otimes_{\ZZ_p} R$-modules of type $(\Lambda_{\{m\}})$, in the terminology of \cite[Def.~3.6]{rapzink96}.  The polarization $\lambda$ makes this into a polarized chain of $\O_F \otimes_{\ZZ_p} R$-modules in a natural way.  By \cite[Th.~2.2]{pappas00}, \'etale-locally on $\Spec R$ there exists an isomorphism of polarized chains between this chain and the chain $\Lambda_{\{m\}} \otimes_{\ZZ_p} R$ which respects the forms on both sides up to a scalar in $R^\times$; here the polarization on $\Lambda_{\{m\}} \otimes_{\ZZ_p} R$ is induced from \eqref{pairing}.  In particular, this chain isomorphism gives an isomorphism of modules
\begin{equation}\label{isom}
   M(A) \isom \Lambda_{-m} \otimes_{\ZZ_p} R,
\end{equation}
and the condition we finally impose on $(A,\iota,\ol\lambda,\ol\eta)$ is that
\begin{enumerate}
\renewcommand{\theenumi}{$*$}
\item\label{cond}
upon identifying the $\Fil^1$ term in the covariant Hodge filtration
\[
   0 \to \Fil^1 \to M(A) \to \Lie A \to 0
\]
with a submodule $\F \subset \Lambda_{-m} \otimes_{\ZZ_p} R$ via \eqref{isom},
the line
\[
   \sideset{}{_R^n}\bigwedge \F \subset W(\Lambda_{-m}) \otimes_{\ZZ_p} R
\]
is contained in $L_{-m}^{n-1,1}(R)$.
\end{enumerate}

Of course the notation here is as in \s\ref{ss:new_cond}.  It is not hard to see that \eqref{cond} is independent of the polarized chain isomorphism used above,
as well as the choice of $\lambda$ and $\pi$, so that it is a well-defined condition.  By the formalism of the local model diagram (see e.g.~\cite[Th.~2.2]{pappas00}), since the scheme $M_{\{m\}}$ is a closed subscheme of $M_{\{m\}}^\naive$ (and also using Remark \ref{rk:newcond==>wedge}), $\S_{C^p}$ is a closed subscheme of the scheme denoted $\A_{C^p}$ in \cite[Def.~6.9]{rapzink96} and \cite[\s2]{pappas00}; in particular, $\S_{C^p}$ is representable by a quasi-projective scheme over $\Spec \O_F$.  Since furthermore $M_{\{m\}}$ and $M_{\{m\}}^\naive$ have the same generic fiber, the same is true of $\S_{C^p}$ and $\A_{C^p}$.  Hence the generic fiber of $\S_{C^p}$ is the base change to $\Spec F$ of a Shimura variety attached to $G$ and the Hermitian data above \cite[(1.19)]{paprap09}.  By Richarz's smoothness result \cite[Prop.~4.16]{arzdorf09} and Theorem \ref{st:main_thm}, we obtain the following.

\begin{cor}
$\S_{C^p}$ is smooth over $\Spec \O_E$.\qed
\end{cor}

Theorem \ref{st:main_thm} can also be applied to certain Shimura varieties for unitary groups attached to CM fields at primes which do not ramify in the totally real subfield, but do in the CM extension.  However we will leave the details for another occasion.

\subsection{The general PEL setting}\label{ss:PEL_setting}

We conclude the paper by explaining how to formulate the condition \eqref{it:K_l} introduced in \s\ref{ss:wedge_power_analogs} in the general PEL setting of Rapoport and Zink's book \cite{rapzink96}.  Given a $\QQ_p$-vector space $V$ and a $\QQ_p$-algebra $R$, we write $V_R := V \otimes_{\QQ_p} R$.

Let $\bigl(F,B,V,\aform,*,\O_B,\{\mu\},\L\bigr)$ be a (local) PEL datum as in \cite{rapzink96}*{\s1.38, Def.\ 3.18}; see also \cite{paprap05}*{\s14}.  This means that
\begin{itemize}
\item
$F$ is a finite product of finite field extensions of $\QQ_p$;
\item
$B$ is a finite semisimple $\QQ_p$-algebra with center $F$;
\item
$V$ is a finite-dimensional left $B$-module;
\item
\aform is a nondegenerate alternating $\QQ_p$-bilinear form $V \times V \to \QQ_p$;
\item
$b \mapsto b^*$ in an involution on $B$ satisfying $\langle bv,w\rangle = \langle v,b^*w\rangle$ for all $v,w \in V$;
\item
$\O_B$ is a $*$-invariant maximal order of $B$;
\item
$\{\mu\}$ is a geometric conjugacy class of cocharacters of the algebraic group $G$ over $\Spec \QQ_p$ whose $R$-points, for any commutative $\QQ_p$-algebra $R$, are
\[
   G(R) = \biggl\{\, g \in GL_B(V_R) \biggm| 
	\begin{varwidth}{\linewidth}
		\centering
		there exists $c(g) \in R^\times$ such that\\
		$\langle gv,gw\rangle = c(g)\langle v,w\rangle$ for all $v,w \in V_R$
	\end{varwidth} 
	\,\biggr\};
\]
and
\item
\L is a self-dual multichain of $\O_B$-lattices in $V$ \cite{rapzink96}*{Defs.\ 3.4, 3.13}.
\end{itemize}

The conjugacy class $\{\mu\}$ is required to satisfy the conditions that for one, hence any, representative $\mu$ defined over one, hence any, extension $K$ of $\QQ_p$, the weights of $\GG_{m,K}$ acting on $V_K$ via $\mu$ are $0$ and $1$, and the composite $c \circ \mu$ ($c$ the similitude character of $G$) is $\id_{\GG_{m,K}}$.  The (local) reflex field $E$ is the field of definition of the conjugacy class $\{\mu\}$, which may also be described as
\[
   \QQ_p\bigl(\,\tr_{\ol \QQ_p} (b \mid V_1) \bigm| b \in B\,\bigr),
\]
where $\ol\QQ_p$ is an algebraic closure of $\QQ_p$ and
\[
   V_1 \subset V_{\ol\QQ_p}
\]
is the weight $1$ subspace of a representative $\mu \in \{\mu\}$.

Rapoport and Zink attach to the above datum the ``naive'' local model in \cite{rapzink96}*{Def.\ 3.27}.  The following is an obvious variant.

\begin{defn} We denote by $M^\naive$ the scheme over $\Spec \O_E$ representing the functor whose values in an $\O_E$-algebra $R$ consist of all pairs of
\begin{itemize}
\item
a functor $\Lambda \mapsto \F_\Lambda$ from \L (regarded as a category in which the morphisms are inclusions of lattices in $V$) to the category of $\O_B \otimes_{\ZZ_p} R$-modules; and
\item
a natural transformation of functors $j_\Lambda\colon \F_\Lambda \to \Lambda \otimes_{\ZZ_p} R$,
\end{itemize}
such that
\begin{enumerate}
\item for all $\Lambda\in\L$, $\F_\Lambda$ is a submodule of $\Lambda \otimes_{\ZZ_p} R$ which is an $R$-direct summand, and $j_\Lambda$ is the natural inclusion $\F_\Lambda \subset \Lambda \otimes_{\ZZ_p} R$;
\item for all $\Lambda \in \L$ and all $b \in B$ which normalize $\O_B$, the composite
\[
   \F_\Lambda^b \subset (\Lambda \otimes_{\ZZ_p} R)^b \xra[\undertilde]b b\Lambda \otimes_{\ZZ_p} R
\]
identifies $(\F_\Lambda)^b$ with $\F_{b\Lambda}$; here for any $\O_B$-module $N$, we denote by $N^b$ the $\O_B$-module whose underlying abelian group is $N$ and whose $\O_B$-action is given by $x \cdot n = b^{-1}x b n$, for $x \in \O_B$ and $n\in N$;
\item for all $\Lambda \in \L$, under the perfect pairing $(\Lambda \otimes_{\ZZ_p} R) \times \bigl(\wh \Lambda \otimes_{\ZZ_p} R \bigr) \to R$ induced by \aform, where $\wh\Lambda$ denotes the dual lattice of $\Lambda$, the submodules $\F_\Lambda$ and $\F_{\wh\Lambda}$ pair to $0$; and
\item (Kottwitz condition) for all $\Lambda \in \L$, there is an equality of polynomial functions over $R$
\[
   \det(b\mid \F_\Lambda) = \det(b \mid V_1), \quad b \in \O_B
\]
(see \cite{rapzink96}*{\s3.23(a)} for the precise meaning of this), where $V_1 \subset V_{\ol \QQ_p}$ is as above.
\end{enumerate}
\end{defn}

Since $V$ admits a nondegenerate symplectic form, its $\QQ_p$-dimension is even, say equal to $2n$.  Choose any $\mu \in \{\mu\}$, and let
\[
   V_{\ol\QQ_p} = V_1 \oplus V_0
\]
be the corresponding weight decomposition.  The condition that $c \circ \mu$ is the identity forces both $V_1$ and $V_0$ to be totally isotropic for the form induced by \aform, and therefore both have dimension $n$.  Fix an integer $l$ with $1 \leq l \leq n$.

For $b \in \O_B$, let $\chi_b(T)$ denote the characteristic polynomial of $b$ acting $\ol\QQ_p$-linearly on $V_1$.  Then $\chi_b(T)$ has coefficients in $\O_E$.  Let $\alpha_1,\dotsc,\alpha_d \in \ol\QQ_p$ denote its distinct roots, and write
\[
   \chi_b(T) = \prod_{i = 1}^d (T - \alpha_i)^{m_i} \in \ol\QQ_p[T].
\]
Thus $m_1 + \dotsb + m_d = n$.  Let $V_{\ol\QQ_p,\alpha_i}$ denote the generalized $\alpha_i$-eigenspace for $b$ acting on $V_{\ol\QQ_p}$.  Define
\[
   \tensor[^l]W{_{b,\ol\QQ_p}} := \bigoplus_{\substack{j_1 + \dotsb + j_d = l\\ 
                                             j_1 \leq m_1\\
											 \vdots\\
											 j_d \leq m_d}}
      \biggl(\sideset{}{_{\ol\QQ_p}^{j_1}} \bigwedge V_{\ol\QQ_p,\alpha_1} \otimes \dotsb \otimes \sideset{}{_{\ol\QQ_p}^{j_d}}\bigwedge V_{\ol\QQ_p,\alpha_d}\biggr).
\]
Then $\tensor[^l]W{_{b,\ol\QQ_p}}$ is naturally a subspace of $\bigwedge_{\ol\QQ_p}^l\! V_{\ol\QQ_p}$, and it is defined over $E$ for the same reason that $\chi_b(T)$ is.  Let $\tensor[^l]W{_b}$ denote its descent to a subspace of $\bigwedge_E^l V_E$.  Define
\[
   \tensor[^l]W{} := \bigcap_{b \in \O_B} \tensor[^l]W{_b} \subset \sideset{}{_E^l}\bigwedge V_E.
\]

We can now give our formulation of condition \eqref{it:K_l} in the present setting, which we continue to denote by \eqref{it:K_l_PEL}.  For $\Lambda \in \L$, let
\[
   \tensor[^l]W{}(\Lambda) := \tensor[^l]W{} \cap \sideset{}{_{\O_E}^l} \bigwedge (\Lambda \otimes_{\ZZ_p} \O_E)
\]
(intersection in $\bigwedge_E^l V_E$).
For $R$ an $\O_E$-algebra and $(\F_\Lambda \subset \Lambda \otimes_{\ZZ_p} R)_{\Lambda \in \L}$ an $R$-point on $M^\naive$, the condition is that
\begin{enumerate}
\renewcommand{\theenumi}{K$_l$}
\item\label{it:K_l_PEL}
for all $\Lambda$, the subbundle $\bigwedge_\Lambda^l \F_\Lambda \subset \bigwedge_R^l (\Lambda \otimes_{\ZZ_p} R)$ is contained in
\[
   \im \Bigl[ \tensor[^l]W{}(\Lambda) \otimes_{\O_E} R \to \sideset{}{_R^l}\bigwedge (\Lambda \otimes_{\ZZ_p} R) \Bigr].
\]
\end{enumerate}

As in \s\ref{ss:wedge_power_analogs}, conditions \eqref{it:K_l_PEL} for $l = 1,\dotsc$, $n$ are closed conditions on $M^\naive$.  On the generic fiber, the Kottwitz condition implies all of them, and is equivalent to (K$_n$).  In general, the analog of Lemma \ref{st:K_n==>Kottwitz} also holds.

\begin{lem}\label{st:whatevs}
Condition $(\mathrm{K}_n)$ implies the Kottwitz condition on $M^\naive$.
\end{lem}

\begin{proof}
Let $R$ be an $\O_E$-algebra and $(\F_\Lambda)_\Lambda$ an $R$-point on $M^\naive$.  We use that the Kottwitz condition holds if and only if for some $\ZZ_p$-basis $b_1,\dotsc,b_r$ of $\O_B$, for all $i$ and all $\Lambda$, the characteristic polynomial of $b_i$ acting $R$-linearly on $\F_\Lambda$ equals $\chi_{b_i}(T)$ in $R[T]$.  For this we may assume that $b_1,\dotsc,b_r$ all act semisimply on $V$, and then the argument is the same as in the proof of Lemma \ref{st:K_n==>Kottwitz}.
\end{proof}

In the special case that $F = B$ is a ramified quadratic extension of $F_0 = \QQ_p$, $V = F^n$, $*$ is the nontrivial element of $\Gal(F/\QQ_p)$, and $\aform$ is the alternating form defined in \s\ref{ss:setup}, condition \eqref{it:K_l_PEL} as defined here reduces to the version in \s\ref{ss:wedge_power_analogs}.  More precisely, continuing to use the notation of \s\ref{ss:setup}, $G$ is the group $GU(\phi)$, which splits over $F$ via the map
\[
   G_F \xra[\undertilde]{(\varphi,c)} GL_n \times \GG,
\]
where as above $c$ is the similitude character, and $\varphi\colon G_F \to GL_n$ is the map given on matrix entries (here regarding $G$ as a subgroup of $\Res_{F/\QQ_p} GL_n$) by
\[
   \xymatrix@R=0ex{
      F \otimes_{\QQ_p} R \ar[r]  &  R\\
	  x \otimes y \ar@{|->}[r]  &  xy
   }
\]
for an $F$-algebra $R$.  Take $\{\mu\}$ to be the geometric conjugacy class of the cocharacter
\[
   t \mapsto \bigl(\diag(\underbrace{t,\dotsc,t}_s,\underbrace{1,\dotsc,1}_r), t\bigr)
\]
of $GL_n \times \GG_m$.  Let \L be the chain of $\O_F$-lattices $\Lambda_i$ for $i \in \pm I + n\ZZ$.  Then the naive local model attached to this PEL datum is the one defined in \s\ref{ss:naiveLM}, and condition \eqref{it:K_l_PEL} is the one defined in \s\ref{ss:wedge_power_analogs}, since the subspace $\tensor[^l]W{^{r,s}}$ defined there equals $\tensor[^l]W{}$, as one readily checks.


\begin{bibdiv}
\begin{biblist}
   

\bib{arzdorf09}{article}{
  author={Arzdorf, Kai},
  title={On local models with special parahoric level structure},
  journal={Michgan Math. J.},
  volume={58},
  date={2009},
  number={3},
  pages={683\ndash 710},
}

\bib{goertz01}{article}{
  author={G{\"o}rtz, Ulrich},
  title={On the flatness of models of certain Shimura varieties of PEL-type},
  journal={Math. Ann.},
  volume={321},
  date={2001},
  number={3},
  pages={689--727},
  issn={0025-5831},
}

\bib{goertz03}{article}{
  author={G{\"o}rtz, Ulrich},
  title={On the flatness of local models for the symplectic group},
  journal={Adv. Math.},
  volume={176},
  date={2003},
  number={1},
  pages={89--115},
}

\bib{kisinpappas?}{article}{
  author={Kisin, M.},
  author={Pappas, G.},
  title={Integrals models of Shimura varieties with parahoric level structure},
  status={in preparation},
}

\bib{kott92}{article}{
   author={Kottwitz, Robert E.},
   title={Points on some Shimura varieties over finite fields},
   journal={J. Amer. Math. Soc.},
   volume={5},
   date={1992},
   number={2},
   pages={373--444},
}

\bib{pappas00}{article}{
  author={Pappas, Georgios},
  title={On the arithmetic moduli schemes of PEL Shimura varieties},
  journal={J. Algebraic Geom.},
  volume={9},
  date={2000},
  number={3},
  pages={577--605},
  issn={1056-3911},
}

\bib{paprap05}{article}{
  author={Pappas, G.},
  author={Rapoport, M.},
  title={Local models in the ramified case. II. Splitting models},
  journal={Duke Math. J.},
  volume={127},
  date={2005},
  number={2},
  pages={193--250},
  issn={0012-7094},
}

\bib{paprap09}{article}{
  author={Pappas, G.},
  author={Rapoport, M.},
  title={Local models in the ramified case. III. Unitary groups},
  journal={J. Inst. Math. Jussieu},
  date={2009},
  volume={8},
  number={3},
  pages={507--564},
}

\bib{prs13}{article}{
  author={Pappas, G.},
  author={Rapoport, M.},
  author={Smithling, B.},
  title={Local models of Shimura varieties, I. Geometry and combinatorics},
  book={ title={Handbook of Moduli, Volume III}, editor={Farkas, G.}, editor={Morrison, I.}, series={Advanced Lectures in Mathematics}, volume={26}, publisher={International Press}, address={Somerville, MA, USA}, },
  date={2013},
  pages={135--217},
}

\bib{pappaszhu13}{article}{
  author={Pappas, G.},
  author={Zhu, X.},
  title={Local models of Shimura varieties and a conjecture of Kottwitz},
  journal={Invent. Math.},
  volume={194},
  number={1},
  pages={147--254},
  date={2013},
}

\bib{rapoportsmithlingzhang?}{article}{
  author={Rapoport, M.},
  author={Smithling, B.},
  author={Zhang, W.},
  title={On arithmetic transfer: conjectures},
  status={in preparation},
}

\bib{rapzink96}{book}{
  author={Rapoport, M.},
  author={Zink, Th.},
  title={Period spaces for $p$-divisible groups},
  series={Annals of Mathematics Studies},
  volume={141},
  publisher={Princeton University Press},
  place={Princeton, NJ, USA},
  date={1996},
  pages={xxii+324},
  isbn={0-691-02782-X},
  isbn={0-691-02781-1},
}

\bib{sm11b}{article}{
  author={Smithling, Brian},
  title={Topological flatness of orthogonal local models in the split, even case. I},
  journal={Math. Ann.},
  volume={350},
  date={2011},
  number={2},
  pages={381--416},
}

\bib{sm11d}{article}{
  author={Smithling, Brian},
  title={Topological flatness of local models for ramified unitary groups. I. The odd dimensional case},
  journal={Adv. Math.},
  volume={226},
  number={4},
  date={2011},
  pages={3160--3190},
}

\bib{sm14}{article}{
  author={Smithling, Brian},
  title={Topological flatness of local models for ramified unitary groups. II. The even dimensional case},
  journal={ J. Inst. Math. Jussieu},
  volume={13},
  number={2},
  date={2014},
  pages={303--393},
}

\bib{sm-decon}{article}{
  author={Smithling, Brian},
  title={Orthogonal analogs of some schemes considered by De Concini},
  status={in preparation},
}



\end{biblist}
\end{bibdiv}
\end{document}